\documentclass{amsart}
\usepackage{amsmath,amsthm,amscd,amssymb}
\usepackage{mathrsfs}
\usepackage{graphicx}

\usepackage{hyperref}
\usepackage{xcolor}
\hypersetup{
colorlinks=green,
citebordercolor=green,
linkbordercolor=green,
urlbordercolor=green
}

\usepackage{enumerate}
\usepackage{color}
\usepackage{amsmath, amssymb,color,amsthm,graphicx}
\usepackage{lineno}

\theoremstyle{plain}
\newtheorem{thm}{Theorem}[section]

\newtheorem{lem}[thm]{Lemma}

\newtheorem{mtheorem}{Theorem}

\theoremstyle{definition}
\newtheorem{dfn}[thm]{Definition}

\newtheorem{remark}[thm]{Remark}
\numberwithin{equation}{section}
\numberwithin{figure}{section}

\def\bb{\mathbb{B}}
\def\vv{\mathbb{V}}
\def\gg{\mathbb{G}}

\def\hh{\mathbb{H}}
\def\rr{\mathbb{R}}
\def\bx{{\boldsymbol{x}}}
\def\by{{\boldsymbol{y}}}
\def\bv{{\boldsymbol{v}}}
\def\bu{{\boldsymbol{u}}}
\def\ba{{\boldsymbol{a}}}

\def\wh{\widehat}
\def\wt{\widetilde}
\def\too{\longrightarrow}
\def\vp{\varphi}
\def\ve{\varepsilon}
\def\ul{\underline}
\def\ol{\overline}

\def\cs{{\mathrm{cs}}}
\def\rs{{\mathrm{s}}}
\def\ru{{\mathrm{u}}}
\def\cu{{\mathcal{U}}}
\def\cf{{\mathcal{F}}}
\def\diff{\mathrm{Diff}}

\def\eset{\emptyset}
\def\la{\langle}
\def\ra{\rangle}

\title[$\boldsymbol{C^1}$-robust strong pluripotency for blender-horseshoes]{$\boldsymbol{C^1}$-robust  strong pluripotency for blender-horseshoes}

\date{\today}



\author{Teruhiko Soma}
\address[Teruhiko Soma]{Department of Mathematical Sciences, Tokyo Metropolitan University, 
Minami-Ohsawa 1-1, Hachioji, Tokyo 192-0397, Japan}
\email{tsoma@tmu.ac.jp}

\author{Shuntaro Tomizawa}
\address[Shuntaro Tomizawa]{Graduate School of Mathematical Sciences, The University of Tokyo, 3-8-1 Komaba, Meguro, Tokyo, 153-8914, Japan}
\email{tomizawa-s@g.ecc.u-tokyo.ac.jp}



\begin{document}

\maketitle

\begin{abstract}
Suppose that $M$ is a closed manifold of dimension greater than two.
We show that there exists a $C^1$-diffeomorphism $f_0:M\longrightarrow M$ with 
a wild affine blender-horseshoe $\Lambda_{f_0}$ 
such that any element $f$ of $\mathrm{Diff}^1(M)$ sufficiently $C^1$-close to $f_0$ is strongly pluripotent for the continuation $\Lambda_f^{(\mathrm{mj})}$ of $\Lambda_{f_0}^{(\mathrm{mj})}$, 
where $\Lambda_{f_0}^{(\mathrm{mj})}$ is the dense subset of $\Lambda_{f_0}$ 
consisting of elements with majority condition.
\end{abstract}

\section{Introduction}

The notion of pluripotency was introduced into dynamical systems in
\cite{KNS}. The term ``pluripotency'' is widely used in physiology and
related fields to describe the ability of a system to transition from an
undifferentiated state to multiple differentiated states governed by its
internal dynamics. In the setting of dynamical systems, pluripotency provides
a mechanism by which the behavior of orbits starting from a hyperbolic
invariant set can be realized stochastically and observably after arbitrarily
small perturbations of the dynamics.

Let $M$ be a closed manifold equipped with a Riemannian metric, and let
$\operatorname{Diff}^{r}(M)$, $r\geq 1$, denote the space of
$C^{r}$-diffeomorphisms of $M$ endowed with the $C^{r}$-topology. For
$f\in\operatorname{Diff}^{r}(M)$ and $x\in M$, we consider the sequence of
empirical measures
\begin{equation}
 \delta_{x,f}^{n}
 =
 \frac{1}{n}\sum_{i=0}^{n-1}\delta_{f^{i}(x)},
 \label{eqn_empirical}
\end{equation}
where $\delta_{f^{i}(x)}$ denotes the Dirac measure supported at $f^{i}(x)$.
Let $\mathcal{P}(M)$ be the space of Borel probability measures on $M$ with
the weak$^*$ topology. We say that the $f$-forward orbit
$(f^{n}(x))_{n\geq 0}$ has historic behavior if the sequence
$(\delta_{x,f}^{n})_{n\geq 1}$ has more than one accumulation point in
$\mathcal{P}(M)$. The notion of historic behavior was introduced by Ruelle
\cite{R01}. One of the main problems concerning historic behavior is
Takens' Last Problem \cite{T08}, which asks whether there exists a persistent
class of smooth dynamical systems for which the set of initial points with
historic behavior has positive Lebesgue measure.

A pioneering contribution to Takens' Last Problem was made by Colli and
Vargas \cite{CV01}. Their model consists of a two-dimensional diffeomorphism
$f$ with a wild horseshoe $\Lambda$. They showed that $f$ admits a wandering
domain $D$ such that every $x\in D$ has historic behavior. Here a wandering
domain is a non-empty connected open set $D$ such that
$f^{i}(D)\cap f^{j}(D)=\emptyset$ whenever $i,j\in\mathbb{Z}$ and $i\neq j$.
In \cite{KNS23}, the method of \cite{CV01} was adapted to wild
blender-horseshoes, yielding analogous phenomena for diffeomorphisms in
dimension three or higher. Further developments in various persistent
classes of smooth dynamical systems were obtained in
\cite{KS17,LR17,B22,BB23,JM24}.

Pluripotency describes orbit behavior in considerably greater detail than
historic behavior. We recall the precise definition in
Definition~\ref{dfn1}. Roughly speaking, for a prescribed point $x$ in a
hyperbolic invariant set, one seeks a perturbation of the dynamics and a
positive Lebesgue measure set of initial points whose statistical, or in the
strong version geometrical, orbit averages asymptotically follow those of
$x$. The two corresponding orbits need not remain close at all times: they
may separate repeatedly, while spending asymptotically dominant portions of
time close to each other; see \cite[Proposition~1.3]{KNS}. This behavior is
formulated in terms of describability in Definition~\ref{describable}.
Although the terminology of pluripotency was not available at the time, the
two-dimensional model of \cite{CV01} and the higher-dimensional model of
\cite{KNS23} can be viewed retrospectively from this perspective.

We now turn to robustness of strong pluripotency. In the $C^{r}$-category,
$r\geq 2$, robust strong pluripotency for the Colli--Vargas model was
established in \cite{KLNSV}: there exists a neighborhood $\mathcal{U}_{0}$
of the model $f_{0}$ such that every $f\in\mathcal{U}_{0}$ is strongly
pluripotent for the continuation $\Lambda_{f}$ of the horseshoe
$\Lambda_{f_{0}}$. In higher dimension, \cite[Theorem~B]{KNS} gives a
corresponding robust result for the wild affine blender-horseshoe introduced
in \cite{KNS23}. In that case, however, strong pluripotency is obtained not
for the whole horseshoe $\Lambda_{f}$ but for the dense subset
$\Lambda_{f}^{(\mathrm{mj})}$ consisting of points satisfying the majority
condition, see Section~\ref{S_Pluri}.

Recent work \cite{KLNS} shows that the restriction to a proper subset is
essential in the present affine blender-horseshoe setting: for a model of the
same type, there is a $C^{r}$-neighborhood, already for $r\geq 1$, in which
no element is strongly pluripotent for the whole horseshoe. Thus the
obstruction to full strong pluripotency is already present in the
$C^{1}$-category. On the other hand, all positive robustness results for
strong pluripotency mentioned above require $r\geq 2$.

The purpose of the present paper is to show that this regularity restriction
can be removed for the positive result on
$\Lambda_{f}^{(\mathrm{mj})}$. More precisely, we prove that the wild affine
blender-horseshoe model of \cite{KNS23} admits a $C^{1}$-neighborhood
$\mathcal{U}_{0}$ such that every $f\in\mathcal{U}_{0}$ is strongly
pluripotent for $\Lambda_{f}^{(\mathrm{mj})}$. To the best of our knowledge,
this gives the first example of a $C^{1}$-open family of diffeomorphisms
exhibiting strong pluripotency on a dense subset of a wild
blender-horseshoe.

Although several parts of the proof are parallel to the $C^{r}$ argument,
$r\geq 2$, in \cite{KNS}, a substantial difficulty appears in the
$C^{1}$-category. The proof in \cite{KNS} uses curvature to control a
sequence $(D_{k})_{k\geq j_{0}}$ of pairwise disjoint solid cylinders
satisfying
\[
 g^{\widehat{n}_{k}+1}(D_{k})\subset D_{k+1}\quad\text{for $k\geq j_{0}$}.
\]
This immediately implies that $\operatorname{Int}D_{j_{0}}$ is a wandering
domain of $g$; see Figure~\ref{f_C1Cr}(a), as well as
\cite[Figure~11]{KLNSV} and \cite[Figure~9.3]{KNS}. 
\begin{figure}[hbtp]
\centering
\scalebox{0.6}{\includegraphics[clip]{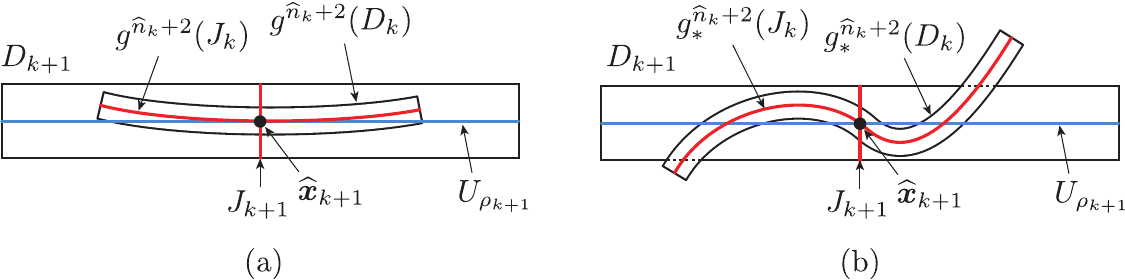}}
\caption{(a) The case  of $C^r$ $(r\geq 2)$.\quad (b) The case of $C^1$.\\
In both cases, $J_{k+1}$ is the core of the solid cylinder $D_{k+1}$ and $U_{\rho_{k+1}}$ is the disk 
centered at $\wh\bx_{k+1}$ of radius $\rho_{k+1}$.}
\label{f_C1Cr}
\end{figure}
In the
$C^{1}$-category, however, curvature is no longer available, and one cannot
exclude the configuration illustrated in Figure~\ref{f_C1Cr}(b). The main
new ingredients used to overcome this difficulty are an estimate of the
deviation of invariant foliations, developed in
Subsection~\ref{ss_deviation}, and a new perturbation procedure by pressing
operations, introduced in Section~\ref{S_second_perturb}.

\section{Pluripotency and main theorem}\label{S_Pluri}

As in Introduction, let $M$ be a closed manifold and $r\geq 1$.
We consider the \emph{first Wasserstein metric} $d_W$ on $\mathcal{P}(M)$ defined as
\[
d_W(\mu ,\nu )= \sup_\varphi \left\vert \int _M \varphi \, d\mu - \int _M \varphi \, d\nu \right\vert
\]
for $\mu$, $\nu\in \mathcal{P}(M)$, 
where the supremum is taken for all Lipschitz functions $\varphi : M \longrightarrow [-1, 1]$ 
with Lipschitz constant bounded by $1$. 
Note that the metric $d_W$ is compatible with the weak$^*$-topology on $\mathcal{P}(M)$.
See \cite{V09} for properties of the metric $d_W$.

Now we recall the concepts of (strong) pluripotency introduced in \cite{KNS}.

\begin{dfn}[Pluripotency]\label{dfn1}
Let $\Lambda_f $ be a uniformly hyperbolic invariant set of $f \in \diff^{r}(M)$.
\begin{enumerate}[(1)] 
\item $f$  is \emph{pluripotent} for  a subset $\Lambda_f'$ of $\Lambda_f$ if, for any $x\in \Lambda_f'$,  
there exist $g \in \diff^r(M)$ arbitrarily $C^r$-close to $f$ and 
a positive Lebesgue measure subset $D_g$ of $M$ such that, 
for any $y\in D_g$, 
\begin{equation}\label{defpl}
\lim _{n\to \infty} d_W( \delta_{y,g}^n, \delta_{x_g,g}^n ) =0,
\end{equation}
where $x_g\in \Lambda_g'$ is the continuation of $x\in \Lambda_f'$. 
\item
$f$ is \emph{strongly pluripotent} for $\Lambda_f'$ if the 
following condition \eqref{defspl} holds instead of \eqref{defpl} for $g$ and $D_g$ as above.
\begin{equation}\label{defspl}
\lim _{n\to \infty} 
\frac{1}{n}\sum_{i=0}^{n-1}
\sup_{y\in D_{g}}\mathrm{dist}(g^i(y),g^i(x_{g}))=0.
\end{equation}
\end{enumerate}
\end{dfn}
We note that \eqref{defspl} implies \eqref{defpl}, while the converse is not true in 
general.
For example, see \cite[Theorem 1.8]{KNS}.

In this paper we consider the case of $\dim M\geq 3$.
In \cite[Theorem B]{KNS}, we showed that, if $2\leq r<\infty$,  
there exists a $C^r$-diffeomorphism $f_0$ on $M$ with a wild affine blender-horseshoe 
which has a neighborhood $\mathcal{U}_0$ in $\diff^r(M)$ such that 
any element $f$ of $\mathcal{U}_0$ is strongly pluripotent for 
the subset $\Lambda_f^{(\rm mj)}$ of 
$\Lambda_f$ consisting of elements with the majority condition.
Here we say that an element $x$ of $\Lambda_f$ satisfies the \emph{majority condition} if 
\begin{equation}\label{eqn_majority}
\liminf_{n\to\infty}\frac{
\#\left\{
j\ ;\ 
n- (3n)^{2/3} \leq j\leq n,\ 
v_{j}=0
\right\}
}{(3n)^{2/3}}\geq \frac12,
\end{equation}
where 
$(v_j)_{j\in \mathbb{Z}}\in \{0,1\}^{\mathbb{Z}}$ is the binary code corresponding to $x$.

\medskip

The following is the main theorem of this paper.

\begin{mtheorem}[$C^1$-robustness of strong pluripotency]\label{mainthm}
Let $M$ be a closed manifold of dimension $\geq 3$.
Then there exist an element $f_{0}$ of $\diff^{1}(M)$ having a wild blender-horseshoe $\Lambda_{f_0}$ 
and an open neighborhood $\cu_0$ of $f_0$ in $\diff^{1}(M)$ such that  every element $f$ of $\mathcal{U}_0$ is strongly pluripotent for the continuation $\Lambda_{f}^{(\mathrm{mj})}$ of $\Lambda_{f_0}^{(\mathrm{mj})}$.
That is, for any $\bx\in \Lambda_f^{(\mathrm{mj})}$, there exist a diffeomorphism $g$ on $M$ arbitrarily $C^1$-close to $f$ and  a positive Lebesgue measure subset $D_{g}$ of $M$ such that \eqref{defspl} holds for the continuation $\bx_{g}\in \Lambda_{g}^{(\mathrm{mj})}$ of $\bx$.
\end{mtheorem}

Since the pressing operation used in the proof of Theorem \ref{mainthm} can be applied also to a $C^1$-neighborhood of the Colli-Vargas model, we do not need arguments based on curvature there.
Therefore the reader might expect that the following problem, posed in \cite{KLNSV}, can be proved as in the case of blender-horseshoe.

\medskip
\emph{
Is there a two-dimensional diffeomorphism which is $C^1$-robustly (strongly)
pluripotent for a horseshoe?
}
\medskip

In contrast to the higher-dimensional case, in dimension two, the process of proving strong pluripotency needs the fact that the stable and unstable thicknesses of the horseshoe vary continuously in a neighborhood of the Colli--Vargas model $f_0$.
However, in the $C^1$-category, Moreira \cite{M11} proved that the continuity of thickness fails.
Therefore the above problem remains open.

\section{Describability and Pluripotency Lemma}\label{S_describable}

In this section, we present the practical necessary and sufficient condition for a diffeomorphism $f$ on 
$M$ to be strongly pluripotent.

A pair $\{\mathbb{U}_0,\mathbb{U}_1\}$ of connected open sets in $M$ with mutually disjoint closures 
 is called a  \emph{coding pair} of a horseshoe $\Lambda_f$ of $f$ if 
\[
\Lambda_f =\bigcap_{i\in \mathbb{Z}} f^{i}(\mathbb U_0 \cup \mathbb U_1)
\]
and the restriction $f|_{\Lambda_f}:\Lambda_f\longrightarrow \Lambda_f$ is topologically conjugate 
to the shift map on $ \{ 0,1\} ^{\mathbb Z}$ by the coding map  
$\mathcal{I}_f: \Lambda_f  \longrightarrow  \{ 0,1\} ^{\mathbb Z}$ defined as 
\[
\left(\mathcal{I}_f(x) \right)_j = v \quad \text{if $f^j(x) \in \mathbb U_{v}$}, 
\]
where $(\mathcal{I}_f(x) )_j$ denotes the $j$-th entry of $\mathcal{I}_f(x)$.

\begin{dfn}[Describability]\label{describable}
Let $\Sigma$ be a subset of $\{0,1\}^{\mathbb N_0}$ and 
$f$ an element of $\mathrm{Diff}^{r}(M)$ with a horseshoe $\Lambda_f$ associated with a coding pair $\{ \mathbb U_0, \mathbb U_1 \}$, 
where $\mathbb N_0=\mathbb{N}\cup \{0\}$.  
We say that $f$ is $\Sigma$-\emph{describable} over $\Lambda_f$  
if any element $\underline{v}=(v_{0}v_{1}v_{2}\ldots\,)$ of $\Sigma$
satisfies the following conditions: 

\begin{enumerate}[({D}1)]
\makeatletter
\renewcommand{\p@enumi}{D}
\makeatother
\item\label{D1}
There exist a strictly increasing sequence $(\alpha _k)_{k\in \mathbb N}$ of non-negative integers and 
integers $\beta _k$ $(k\in \mathbb N)$ with $0\leq \beta _k\leq \alpha _{k+1}-\alpha _k-1$ and  
such that 
\[
\lim _{n\rightarrow \infty} \frac{\#\left\{j\,;\, 0\le j \le n-1,\ j \in \bigcup _{k=1}^\infty \mathbb{I}_k\right\}}{n} =1, 
\]
where $\mathbb{I}_k
=[\alpha_k, \alpha _k+ \beta_k] \cap \mathbb Z$.
\item\label{D2}
There exist an element $g$ of $\mathrm{Diff}^{r}(M)$ arbitrarily $C^r$-close to $f$ and  
a positive Lebesgue measure subset $D_g$ of $M$  
such that  
\[
g^j (D_g) \subset \mathbb U_{v_{j}}
\]
for any  $j\in \bigcup_{k\in \mathbb{N}
} \mathbb I_k$.  
\end{enumerate}
\end{dfn}

The following theorem (Pluripotency Lemma) is often used to decide whether a given dynamical system is strongly pluripotent. 
Although we do not directly cite this theorem in this paper, some argument in Subsection \ref{ss_proof_thmA} is based on ideas from its proof.

\begin{thm}[{\cite[Theorem A]{KNS}}]\label{p-lemma}
Suppose that $f$ is an element of $\mathrm{Diff}^r(M)$ $(r\geq 1)$ with a horseshoe $\Lambda_f$ associated with a coding pair $\{ \mathbb U_0, \mathbb U_1 \}$ and 
$\Sigma$ is a non-empty subset of $\{0,1\}^{\mathbb N_0}$.
Then $f$ is $\Sigma$-describable if and only if 
$f$ is strongly pluripotent for $\mathcal{I}_f^{-1}(\widehat \Sigma)$.
\end{thm}
Here $\widehat \Sigma$ is the subset of $\{0,1\}^{\mathbb{Z}}$ consisting of elements 
$(v_j)_{j\in \mathbb{Z}}$ with $(v_j)_{j\in \mathbb{N}_0}\in \Sigma$.

\section{Fundamental settings}

\subsection{Wild blender-horseshoe}

Now we recall a wild affine blender-horseshoe $\Lambda$ given in \cite{KNS23,KNS}, 
where `\emph{wild}' means that $\Lambda$ has a $C^1$-robust homoclinic tangency.
Let $\mathcal{U}_0$ be an open neighborhood of $f_0$ in $\diff^1(M)$, which 
will be replaced with a smaller neighborhood as needed.

In this paper, we only consider the case of $\dim M=3$.
Even in the case of $\dim M=n>3$, our arguments still work for certain elements of 
$\diff^1(M)$ having a blender-horseshoe $\Lambda$ with $\dim W^u(\Lambda)=\dim W^{\cs}(\Lambda)=1$ and $\dim W^{\mathrm{ss}}(\Lambda)=n-2$.

Let $\lambda_{\rm ss}, \lambda_{\rm cs0}, \lambda_{\rm cs1}$ and $\lambda_{\rm u}$ 
be real positive constants
with
\begin{equation}\label{eqn_eigen_v}
\lambda_{\rm ss}<\lambda_{\rm cs0}<1/2<\lambda_{\rm cs1}<1<\lambda_{\rm cs0}+\lambda_{\rm cs1},\quad 2<\lambda_{\rm u},\quad \lambda_{\rm cs0}\lambda_{\rm cs1}\lambda_{\rm u}^{2}<1.
\end{equation}
We fix a sufficiently small  positive number $\varepsilon_0$ and set 
$$
I_{\varepsilon_{0}}=[-\varepsilon_{0},1+\varepsilon_{0}].
$$
Suppose that the cube 
 $\mathbb{B}=I_{\varepsilon_{0}}^{3}$ is embedded in $M$ and $M$ has a Riemannian metric 
 extending the standard Euclidean metric on $\bb$.
Moreover, a neighborhood of $\bb$ in $M$ has a coordinate 
$(x,y,z)$ extending that on $I_{\ve_0}^3$.
 Consider the sub-blocks of $\mathbb{B}$ defined as 
$$
\mathbb{V}_{0}= [-\varepsilon_{0}, \lambda_{\rm u}^{-1}+\varepsilon_0]\times I_{\varepsilon_{0}}^{2},\quad 
\mathbb{V}_{1}= [1-\lambda_{\rm u}^{-1}-\varepsilon_0,1+\varepsilon_{0}]\times I_{\varepsilon_{0}}^{2}. 
$$
\begin{subequations}
See Figure \ref{f_blender}\,(a).
\begin{figure}[hbtp]
\centering
\scalebox{0.6}{\includegraphics[clip]{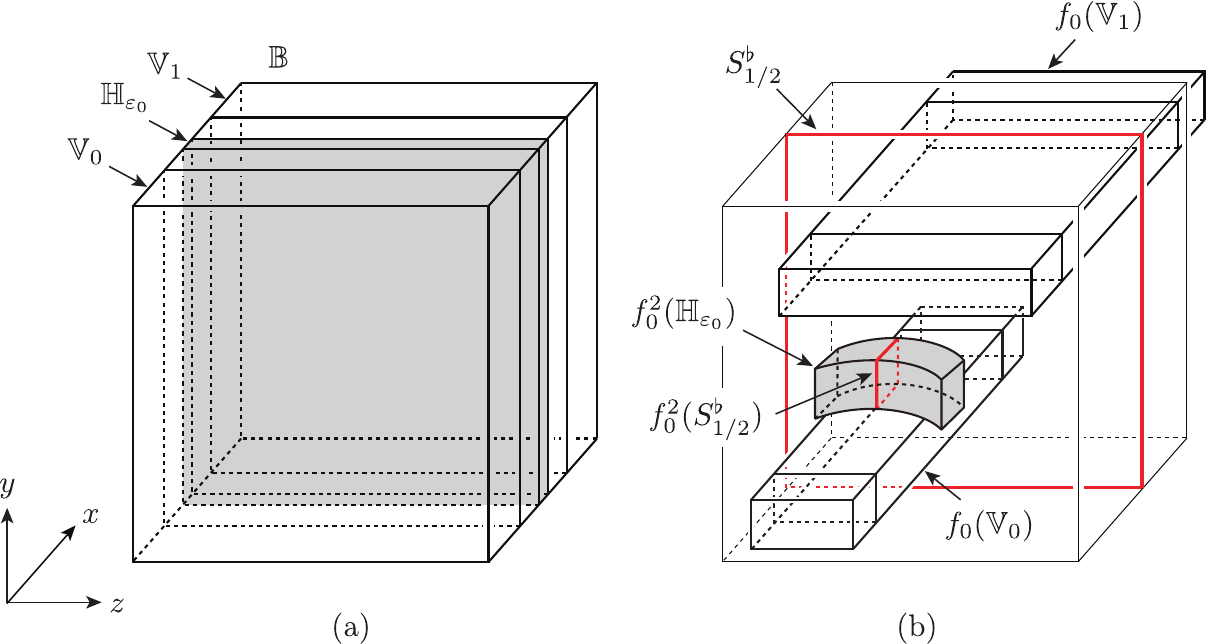}}
\caption{A wild blender-horseshoe.}
\label{f_blender}
\end{figure}
Let $f_{0}$ be a diffeomorphism on $M$ such that $f_{0}|_{\mathbb{V}_{0}\cup \mathbb{V}_{1}}$ is defined as
\begin{equation}\label{eqn_f0V}
f_{0}(x,y,z)=
\begin{cases}
(\lambda_{\mathrm{u}}x,\lambda_{\mathrm{ss}}y,\zeta_0(z))&\text{if}\ (x,y,z)\in \mathbb{V}_{0},\\
(\lambda_{\mathrm{u}}(1-x),-\lambda_{\mathrm{ss}}y+1,\zeta_1(z)) &\text{if}\ (x,y,z)\in \mathbb{V}_{1},
\end{cases}
\end{equation} 
where 
$\zeta_0$ and $\zeta_1$ are the affine maps on $I_{\varepsilon_0}$ given by
\begin{equation}\label{eqn_def_zeta}
\zeta_0(z)=\lambda_{\mathrm{cs} 0}z\quad\text{and}\quad \zeta_1(z)=\lambda_{\mathrm{cs} 1}z+1-\lambda_{\mathrm{cs} 1}.
\end{equation}
\end{subequations}
From our setting, $f_{0}$ has the uniformly hyperbolic set  
\[\Lambda_{f_{0}}=\bigcap_{n\in \mathbb{Z}} f_{0}^n(\mathbb{V}_{0}\cup \mathbb{V}_{1})\]
which belongs to the class of blender-horseshoes, see \cite{BD96,BD12} for details.

Next we suppose another condition to obtain a non-hyperbolic situation.
Let 
$\mathbb{H}_{\varepsilon_0}$ be the subset of $\mathbb{B}$ defined as  
\[
\mathbb{H}_{\varepsilon_0}=[1/2-\varepsilon_0,1/2+\varepsilon_0]\times I_{\varepsilon_0}^2.
\]
\begin{subequations}
For any $(x,y,z)\in \mathbb{H}_{\varepsilon_0}$, 
we suppose that   
\begin{equation}\label{eqn_tang}
f_{0}^{2}(x,y,z)=\left(-a_{1}\Bigl(x-\frac{1}{2}\Bigr)^{2}+a_{2}z+\mu,\ a_{3}\Bigl(y-\frac{1}{2}\Bigr)+\frac{1}{2},\ 
a_{4}\Bigl(x-\frac{1}{2}\Bigr)+\frac{1}{2} \right),
\end{equation}
where $a_{1},a_{2},a_{3},a_{4},\mu$ are real constants satisfying  
\begin{equation}\label{eqn_a_1_4}
a_{1}>0,\ a_2>0,\ 
|a_{3}|<1-2\lambda_{\rm ss}\quad\text{and}\quad a_3a_4<0,
\end{equation} 
and
\begin{equation}\label{eqn_a2_mu}
a_1\ve_0^2+a_2\ve_0<\mu,\quad a_2(1+\ve_0)+\mu<\lambda_{\ru}^{-1}.
\end{equation}
\end{subequations}
See Figure \ref{f_blender}\,(b).
Since $a_2a_3a_4<0$ by \eqref{eqn_a_1_4}, $f_0^2|_{\mathbb{H}_{\varepsilon_0}}$ is orientation preserving.
The condition $|a_{3}|<1-2\lambda_{\rm ss}$ assures that $f_0^2(\hh_{\ve_0})$ 
lies between $f_0(\vv_0)$ and $f_0(\vv_1)$.
The condition \eqref{eqn_a2_mu} implies    
$f_0^2(\hh_{\ve_0})\subset (0,\lambda_{\ru}^{-1})\times I_{\ve_0}^2\subset \vv_0$.
See Figure \ref{f_Swk} below.
By using the distinctive property $\lambda_{\cs 0}+\lambda_{\cs1}>1$ in \eqref{eqn_eigen_v}, we know that $\Lambda_{f_0}$ has a 
robust homoclinic tangency.
See Section 6.2 in \cite{BDV05} for details.

\subsection{Invariant cone-fields and foliations}\label{ss_inv_cone}
For a small $0<\ve<\ve_0$, we consider the following $\mathrm{u}$,  $\mathrm{cs}$-cone-fields on $\mathbb{B}$.
\begin{align*}
\boldsymbol{C}_{\varepsilon}^{\mathrm{u}}(\boldsymbol{x})&=
\left\{
(v^{\mathrm{u}},v^{\mathrm{s}},v^{\mathrm{cs}})\in T_{\boldsymbol{x}}(\mathbb{B})\,;\,
\sqrt{(v^{\mathrm{s}})^2+(v^{\mathrm{cs}})^2}\leq \varepsilon |v^{\mathrm{u}}|
\right\},\\
\boldsymbol{C}_{\varepsilon}^{\mathrm{cs}}(\boldsymbol{x})&=
\left\{
(v^{\mathrm{u}},v^{\mathrm{s}},v^{\mathrm{cs}})\in T_{\boldsymbol{x}}(\mathbb{B})\,;\,
|v^{\mathrm{u}}|\leq\varepsilon\sqrt{(v^{\mathrm{s}})^2+(v^{\mathrm{cs}})^2}
\right\}
\end{align*}
for $\boldsymbol{x}\in\mathbb{B}$.

Now we fix a neighborhood $\mathcal{U}_0$ of $f_0$ in $\mathrm{Diff}^1(M)$, 
which will be replaced by a smaller neighborhood as needed.
So we may assume that, for any $f\in \mathcal{U}_0$, 
$\boldsymbol{C}_{\varepsilon}^{\mathrm{u}}$ is $f$-invariant and 
$\boldsymbol{C}_{\varepsilon}^{\mathrm{cs}}$ is $f^{-1}$-invariant.
We say that a $C^1$-surface $F$ in $\mathbb{B}$ is \emph{adaptable} to $\boldsymbol{C}_{\varepsilon}^{\mathrm{cs}}$ if, for any $\boldsymbol{x}\in F$, 
the tangent plane $T_{\boldsymbol{x}} F$ is contained in $\boldsymbol{C}_{\varepsilon}^{\mathrm{cs}}(\boldsymbol{x})$.
Similarly a $C^1$-arc $\alpha$ in $\mathbb{B}$ is \emph{adaptable} to $\boldsymbol{C}_{\varepsilon}^{\mathrm{u}}$ if, 
for any $\boldsymbol{x}\in \alpha$, $T_{\boldsymbol{x}}\alpha$ is contained in  $\boldsymbol{C}_{\varepsilon}^{\mathrm{u}}(\boldsymbol{x})$.

By the same procedure as in \cite[Subsection 2.4]{PT93},
we can obtain a  
$C^{0}$ stable foliation $\mathcal{F}^{\rm s}_{f}$ on $\mathbb{B}$ 
which is compatible with $W_{\mathrm{loc}}^{\rs}(\Lambda_{f})$ and satisfies the following conditions.
\begin{enumerate}[({F}1)]
\makeatletter
\renewcommand{\p@enumi}{F}
\makeatother
\item
Each leaf of $\mathcal{F}_f^{\mathrm{s}}$ is a $C^1$-surface in $\mathbb{B}$.\label{F1}
\item
The restriction $\mathcal{F}_f^{\mathrm{s}}|_{\mathbb{H}_{\varepsilon_0}}$ consists of flat leaves parallel to the $yz$-plane.\label{F2}
\item
Any leaf of $\mathcal{F}_f^{\mathrm{s}}$ is adaptable to $\boldsymbol{C}_{\varepsilon}^{\mathrm{cs}}$.\label{F3}
\end{enumerate} 
Then $\mathcal{F}_f^{\cs}:=f^{-2}(\mathcal{F}_f^{\mathrm{s}})\cap \hh_{\ve_0}$ is the 
foliation on $\hh_{\ve_0}$.
For $\bx\in \bb$, we denote by $F^{\rs}(\bx)$ the leaf of $\cf_f^{\rs}$ containing 
$\bx$.
Similarly, for $\by\in \hh_{\ve_0}$, $F^{\cs}(\by)$ denotes the leaf of $\cf_f^{\cs}$ 
with $F^{\cs}(\by)\ni \by$.

Let $S_{1/2}^\flat$ be the section $\{x=1/2\}$ of $\hh_{\ve_0}$.
See Figure \ref{f_curv_block}.
For any mutually disjoint leaves $F_0$ and $F_1$ in $\cf_f^{\rs}$, 
the closure of the component of $\bb\setminus F_0\cup F_1$ disjoint 
from the sides $\{-\ve_0,1+\ve_0\}\times I_{\ve_0}^2$ of $\bb$ is called a \emph{u-block}.
For $i=0,1$, the smallest u-block containing $\mathbb{V}_i\cap \Lambda_f$ 
is denoted by $\bb^{\mathrm{u}}(i)$.
For any integer $k\geq 1$, let  
$\underline{w}^{(k)}=w_{1}w_2\dots w_{k-1}w_{k}\in \{0,1\}^{k}$ be a binary code of $k$ entries.
We define the subset  $\mathbb{B}^{\rm u}(\underline{w}^{(k)})$ of $\bb$ by 
\begin{equation}\label{eqn_BBu}
\mathbb{B}^{\rm u}(\underline{w}^{(k)})
=\left\{\boldsymbol{x}\in \mathbb{B}\,;\, f^{i-1}(\boldsymbol{x})\in \mathbb{B}^{\mathrm{u}}(w_i),
i=1,\dots,k\right\},
\end{equation}
which is called the \emph{u-bridge block} associated with $\ul w^{(k)}$.
The subsets $\mathbb{H}_{\underline{w}^{(k)}}$ and $S_{\underline{w}^{(k)}}^\flat$ 
of $\mathbb{B}^{\rm u}(\underline{w}^{(k)})$ are defined as
\begin{equation}\label{eqn_gg_gamma}
\begin{split}
\mathbb{H}_{\underline{w}^{(k)}}&=(f|_{\mathbb{V}_{w_1,f}})^{-1}\circ (f|_{\mathbb{V}_{w_2,f}})^{-1}\circ\cdots \circ (f|_{\mathbb{V}_{w_{k-1},f}})^{-1}\circ (f|_{\mathbb{V}_{w_k,f}})^{-1}(\mathbb{H}_{\varepsilon_0}),\\
S_{\underline{w}^{(k)}}^\flat&=(f|_{\mathbb{V}_{w_1,f}})^{-1}\circ (f|_{\mathbb{V}_{w_2,f}})^{-1}\circ\cdots \circ (f|_{\mathbb{V}_{w_{k-1},f}})^{-1}\circ (f|_{\mathbb{V}_{w_k,f}})^{-1}(S_{1/2}^\flat),
\end{split}
\end{equation}
where $\mathbb{V}_{i,f}$  $(i=0,1)$ is the component of $\mathbb{B}\cap f^{-1}(\mathbb{B})$ contained in $\mathbb{V}_i$.
By \eqref{F2}, $S_{1/2}^\flat$ is a leaf of $\cf_f^{\rs}$ and hence 
$S_{\underline{w}^{(k)}}^\flat$ is also a leaf of $\cf_f^{\rs}|_{\mathbb{B}^{\rm u}(\underline{w}^{(k)})}$.

Since $f$ is sufficiently $C^1$-close to $f_0$, it follows from the construction of $f_0$ that 
one can take $\mu$ satisfying \eqref{eqn_a2_mu} and a binary code $\ul{w}^{(n_0)}$ 
of large length $n_0$ so that 
 $\mathbb{B}^{\rm u}(\underline{w}^{(n_0)})$ is divided into two components by $f^2(\mathrm{Int}\,S_{1/2}^\flat)$.
 If necessary replacing $\ve_0$ with a smaller positive constant, we may also assume that 
$\mathbb{B}^{\rm u}(\underline{w}^{(n_0)})$ is disjoint from the roof 
$f^2(\hh_{\ve_0}\cap \{z=1+\ve_0\})$ of $f^2(\hh_{\ve_0})$.
See Figure \ref{f_Swk}.
\begin{figure}[hbtp]
\centering
\scalebox{0.6}{\includegraphics[clip]{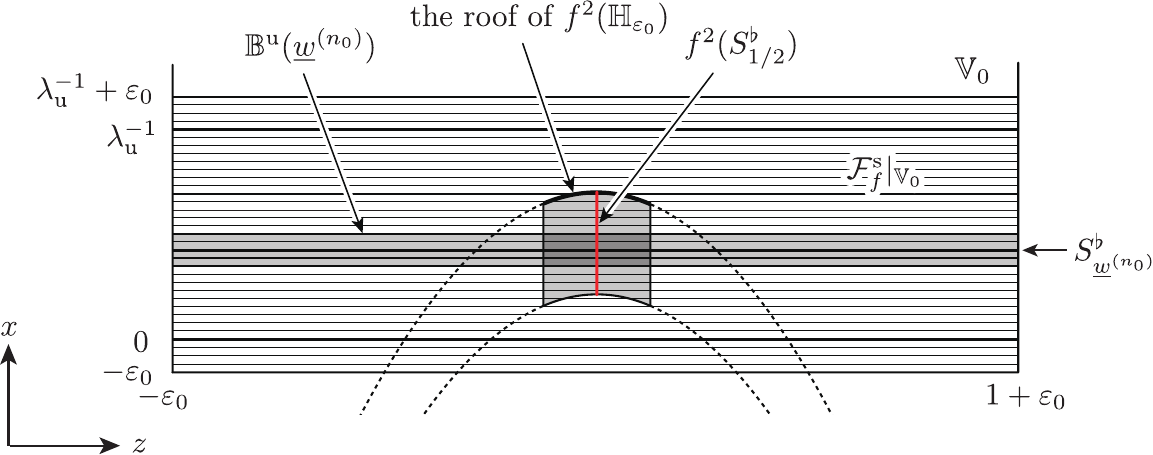}}
\caption{Transition from $\hh_{\ve_0}$ to $\vv_0$.}
\label{f_Swk}
\end{figure}

Let  $\mathbb{H}_{\,[k]}=\bigcup_{\underline{w}^{(k)}\in \{0,1\}^k}\mathbb{H}_{\underline{w}^{(k)}}$ and  
$\mathbb{H}_{\,[\infty]}=\bigcup_{k=0}^\infty \mathbb{H}_{\,[k]}$, where $\mathbb{H}_{\,[0]}=\mathbb{H}_{\varepsilon_0}$.
Let $\mathcal{L}_{(k;\infty)}$ be the 1-dimensional $C^0$-foliation on $\hh_{[k]}$ 
consisting of $C^1$-leaves defined as in \cite[Subsection 4.1]{KNS}, which is adaptable to $\boldsymbol{C}_{\varepsilon}^{\mathrm{u}}$.
Moreover it is shown that the union 
\[
\mathcal{L}_\infty:=\bigcup_{k=0}^\infty\mathcal{L}_{(k;\infty)}
\]
is $f$-\emph{invariant} in the sense that, for any 
leaf $l$ of $\mathcal{L}_{(k+1;\infty)}$, $f(l)$ is a leaf of $\mathcal{L}_{(k;\infty)}$.
From the construction of $\mathcal{L}_{(k;\infty)}$, we also know that, for any $\ve_*>0$ and 
any leaf $l$ of $\mathcal{L}_{(k;\infty)}$, there exists a neighborhood $\mathcal{N}(l)$ of 
$l$ in $\bb$ such that any leaf $l'$ of $\mathcal{L}_{(k;\infty)}$ contained in $\mathcal{N}(l)$ is $C^1$-$\ve_*$-close to $l$.
Even so, the authors are not certain whether $\mathcal{L}_{(k;\infty)}$ is $C^1$-diffeomorphic to a product foliation.

\section{First perturbation}

In this section, we present results on $C^1$-diffeomorphisms obtained by 
arguments similar to those in \cite{KNS} with minor changes.
Arguments applicable only in the $C^1$-category will be treated in subsequent sections.

\subsection{Estimation of eigenvalues}\label{ss_est_eigen}

\begin{subequations}\label{ss_est_eigen}
Suppose that $\lambda_{\sigma}$ with $\sigma\in\{{\rm u,ss, cs0, cs1}\}$ are 
the constants  given in  \eqref{eqn_eigen_v}.
For a sufficiently small $\varepsilon>0$, we write
\begin{equation}\label{abb1}
\underline{\lambda}_{\sigma}=(1-\varepsilon)\lambda_{\sigma},\ 
\ol{\lambda}_{\sigma}=(1+\varepsilon)\lambda_{\sigma}.
\end{equation}
By \eqref{eqn_eigen_v}, one can choose $\ve>0$ so that the following inequalities hold.
\begin{equation}\label{eqn_eigen_v2}
\begin{split}
&\overline{\lambda}_{\rm ss}<\ul\lambda_{\rm cs0}<\ol\lambda_{\rm cs0}<1/2<\underline\lambda_{\rm cs1}<\ol\lambda_{\rm cs1}<1<\underline\lambda_{\rm cs0}+\underline\lambda_{\rm cs1},\quad 2<\underline\lambda_{\rm u},\\
&\ol\lambda_{\rm cs0}\ol\lambda_{\rm cs1}\ol\lambda_{\ru}^{\,4}\ul\lambda_{\ru}^{-2}<1.
\end{split}
\end{equation}
\end{subequations}
By \eqref{eqn_f0V}, we may take $\mathcal{U}_0$ so that, for any $f\in \mathcal{U}_0$, 
$Df(\boldsymbol{x})$ is arbitrarily $C^{0}$-close to the diagonal matrix
\[
Df_0(\boldsymbol{x})=\mathrm{diag}((-1)^i\lambda_{\mathrm{u}},(-1)^i\lambda_{\mathrm{ss}},\lambda_{\mathrm{cs} i})
\]
for $\boldsymbol{x}\in \mathbb{V}_{i}$ $(i=0,1)$.
So one can suppose that, for any $\bx\in \mathbb{V}_{0,f}\cup \mathbb{V}_{1,f}$ 
and any unit vector $\bv$ of $\boldsymbol{C}_\ve^{\ru}(\bx)$,
\begin{equation}\label{eqn_lamDf}
\underline{\lambda}_{\ru}<\|Df(\bx)\bv\|<\ol\lambda_{\ru}.
\end{equation}
It follows that, for any $C^1$-curve $l$ in $\mathbb{B}^{\mathrm{u}}(\underline{w}^{(k)})$ adaptable to the cone-field $\boldsymbol{C}_{\varepsilon}^{\mathrm{u}}$, 
\begin{equation}\label{eqn_lamfkl}
\ol{\lambda}_{\mathrm{u}}^{\,-k}|f^k(l)|<
|l|<\ul{\lambda}_{\mathrm{u}}^{-k}|f^k(l)|, 
\end{equation}
where the length of a $C^1$-arc $\alpha$ in $\bb$ is denoted by $|\alpha|$. 
We may also assume that, for any $\bx\in \mathbb{V}_{i,f}$ $(i=0,1)$ and any  unit vector $\bv_i$ 
tangent to $F^{\rs}(\bx)$ at $\bx$, 
\begin{equation}\label{eqn_lambda_e/2}
\|Df(\bx)\bv_i\|<\ol\lambda_{\cs i}.
\end{equation}

\subsection{A sequence of binary codes}\label{ss_binary_code}
We denote the length of a finite binary code $\ul{w}$ by $|\ul{w}|$.
For the binary code $\ul{w}^{(n_0)}$ given in Subsection \ref{ss_inv_cone} and an integer $m\geq 0$, 
we consider a binary code $\ul{w}=\ul{w}^{(n_0+m)}$ represented as $\ul{w}^{(n_0)}\ul{u}^{(m)}$, 
where $\ul{u}^{(m)}$ is a binary code with $|\ul{u}^{(m)}|=m$.
Note that $\bb^{\ru}(\ul{w})$ is divided into two components by $f^2(\mathrm{Int}\,S_{1/2}^\flat)$.
We say that 
$\mathbb{U}_{\underline{w}}^{\mathrm{cs}}:=f^{-2}(\mathbb{B}^{\mathrm{u}}(\underline{w}))\cap \mathbb{H}_{\varepsilon_0}$ 
is the \emph{cs-curved block} for $\ul{w}$ and 
and $\wh\Sigma_{\ul{w}}^\flat:=\mathbb{U}_{\ul w}^{\mathrm{cs}}\cap S_{1/2}^\flat$ is the \emph{section} of $\mathbb{U}_{\ul w}^{\mathrm{cs}}$.
See Figure \ref{f_curv_block}.
\begin{figure}[hbtp]
\centering
\scalebox{0.6}{\includegraphics[clip]{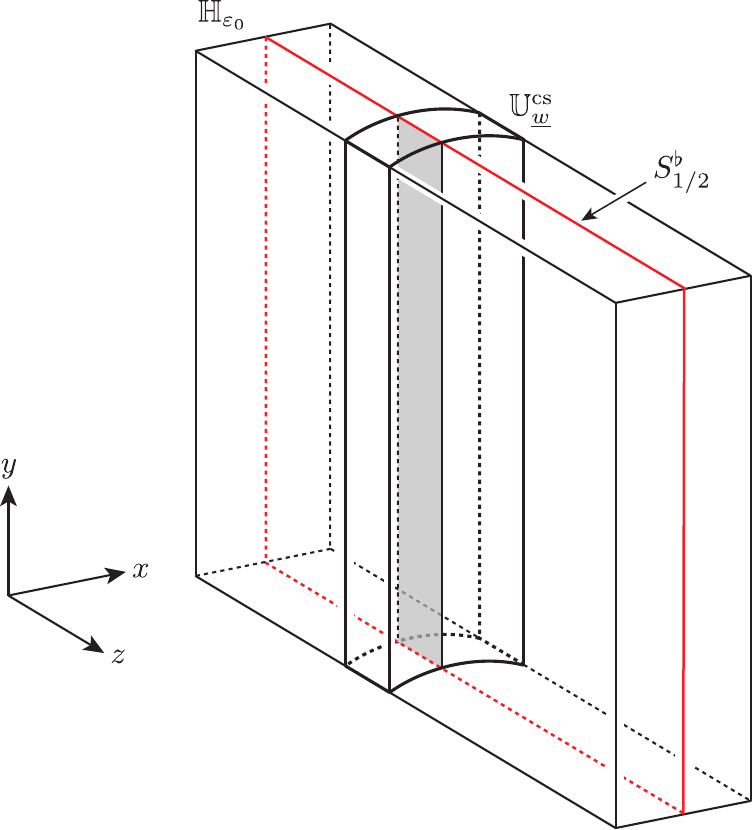}}
\caption{The shaded rectangle represents $\wh\Sigma_{\ul w}^\flat$.}
\label{f_curv_block}
\end{figure}

The following lemma is proved by arguments quite similar to those in \cite[Lemma 8.1]{KNS}, 
where all positive integers are independent of $k$.
However we use here the $C^1$-surfaces $S_{\ul w}^\flat$  
instead of the topological surfaces $S_{\ul w}^{\mathrm{cs}}$ in \cite{KNS}.
We should note that the distinctive condition $\underline\lambda_{\rm cs0}+\underline\lambda_{\rm cs1}>1$ of \eqref{eqn_eigen_v2} is crucial in the proof of the lemma.

\begin{lem}\label{lem-3-1}
Fix an integer $L\geq 4$ arbitrarily.
There exists a sequence $(\underline{w}^{(n_0+Lk)})_{k\geq 1}$ of binary codes 
extending $\underline{w}^{(n_0)}$ with 
 $|\underline{w}^{(n_0+Lk)}|=n_0+Lk$ for some positive integer $n_0$ 
 and  there exists a binary code $\underline{\widehat w}_k$
for an arbitrarily chosen binary code $\underline{u}_k$ of finite length satisfying the following conditions.
\begin{enumerate}[\rm (1)]
\item\label{lem-3-1(0)}
$\bb^{\ru}(\ul{w}^{(n_0+Lk)})\cap \bb^{\ru}(\ul{w}^{(n_0+Ll)})=\emptyset$ for any non-negative integers $k,l$ with $k\neq l$ and $\bb^{\ru}(\ul{w}^{(n_0+Ll)})$ is closer to the bottom $\{x=-\varepsilon_0\}$ of $\bb$ if $l>k$.
See Figure \ref{f_Bun0}.
\item\label{lem-3-1(1)}
$\underline{\widehat w}_k$ 
is represented as $\underline{w}^{(n_0+Lk)}\underline{u}_k\underline{\iota}_k\underline{\gamma}^{(m_k)}$, 
where 
$\underline{\iota}_k$ and $\underline{\gamma}^{(m_k)}$ are binary codes given as follows. 
\begin{itemize}
\setlength{\itemindent}{-16pt}
\item
$|\underline{\iota}_k|\leq \mu_0$ for some positive integer $\mu_0$ (possibly $\underline{\iota}_k=\emptyset$). 
\item
$\ul\gamma^{(m_k)}$ is represented as 
$\ul\gamma^{(m_k)}=\gamma_{m_k}\gamma_{m_k-1}\cdots \gamma_2\gamma_1$  
for some $0<m_k\leq N_0+N_1k$, where $N_0$ and $N_1$ are  properly chosen positive integers.
\end{itemize}
\item \label{lem-3-1(2)}
$f^{|\widehat{\underline{w}}_k|}(S_{\widehat{\underline{w}}_k}^\flat)$ is contained in $\wh\Sigma_{\underline{w}^{(n_0+L(k+1))}}^\flat$, see \eqref{eqn_gg_gamma} for 
$S_{\widehat{\underline{w}}_k}^\flat$.
\end{enumerate}
\end{lem}
\begin{figure}[hbtp]
\centering
\scalebox{0.6}{\includegraphics[clip]{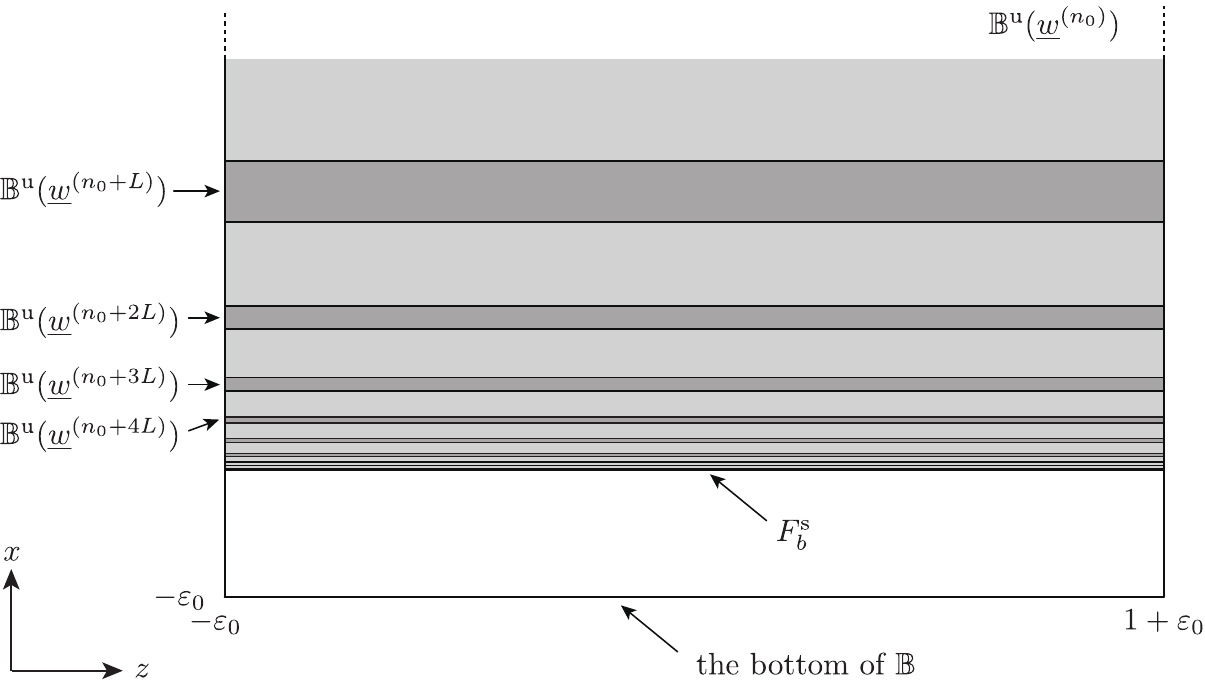}}
\caption{
$B^{\ru}(\ul w^{(n_0)})$ in this figure is not drawn to scale. 
The figure is intended to show the locations of $B^{\ru}(\ul w^{(n_0+L)})$, 
$B^{\ru}(\ul w^{(n_0+2L)}), \dots$ in $B^{\ru}(\ul w^{(n_0)})$. 
In reality, $B^{\ru}(\ul w^{(n_0)})$ is a much thinner block.
$\bigl(\bb^{\ru}(\ul{w}^{(n_0+Lk)})\bigr)_{k\geq 1}$ converges to the bottom $F_{b}^{\rs}$ 
of $\bb^{\ru}(\ul{w}^{(n_0)})$ in the Hausdorff metric, which is a leaf of $\cf_f^{\rs}$.
}
\label{f_Bun0}
\end{figure}

We refer to \cite{KNS} for the precise values of $n_0$, $L$, $\mu_0$, $N_0$ and $N_1$.
In \cite[Subsection 3.3]{KNS}, it is shown that $\ul w^{(n_0+Lk)}$ in this lemma is 
a binary code extending $\ul w^{(n_0+k)}$ and hence extending $\ul w^{(n_0)}$.
This implies that both $\ul{w}^{(n_0+Lk)}$ and $\ul{w}^{(n_0+Ll)}$ are extensions of $\ul{w}^{(n_0)}$ for any $k,k'$ with $0<k<k'$.
However, they have distinct $(n_0+k)$-th entries, see (3.9) in \cite{KNS}, so 
$\ul{w}^{(n_0+Lk')}$ is not an extension of $\ul{w}^{(n_0+Lk)}$ in any case.

\begin{remark}\label{r_free_u}
The fact that $\ul{u}_k$ can be chosen arbitrarily plays a crucial role in proving 
Theorem \ref{mainthm}.
\end{remark}

We set $|\widehat{\underline{w}}_k|=\widehat n_k$ for short.
From now on, we only consider the case of $|\ul{u}_k|=k^2$.
Then, by Lemma \ref{lem-3-1}, 
\begin{equation}\label{eqn_hatn0}
\widehat n_k=n_0+Lk+|\underline{u}_k|+|\underline{\iota}_k|+m_k=k^2+O(k).
\end{equation}

\subsection{Construction of a pseudo-orbit of $f$}\label{ss_const_p-o}

From the shape of $\cf_f^{\cs}$, for any leaf $l$ of $\mathcal{L}_{(0,\infty)}$, 
there exists a leaf of $\cf_f^{\cs}$ tangent to $l$.
In fact, when leaves of $\cf_f^{\rs}$ travel from the right side $\{z=1+\ve_0\}$ to the left side $\{z=-\ve_0\}$ of $\hh_{\ve_0}$, any first encountering point of $l$ and a leaf of $\cf_f^{\rs}$ is a tangency point.
In contrast to the $C^r$-case $(r\geq 2)$, it may occur that more than one leaf is tangent to a single 
$l$ or a single leaf of $\cf_f^{\cs}$ is tangent to $l$ in more than one points.
See Figure \ref{f_taul}.
\begin{figure}[hbtp]
\centering
\scalebox{0.6}{\includegraphics[clip]{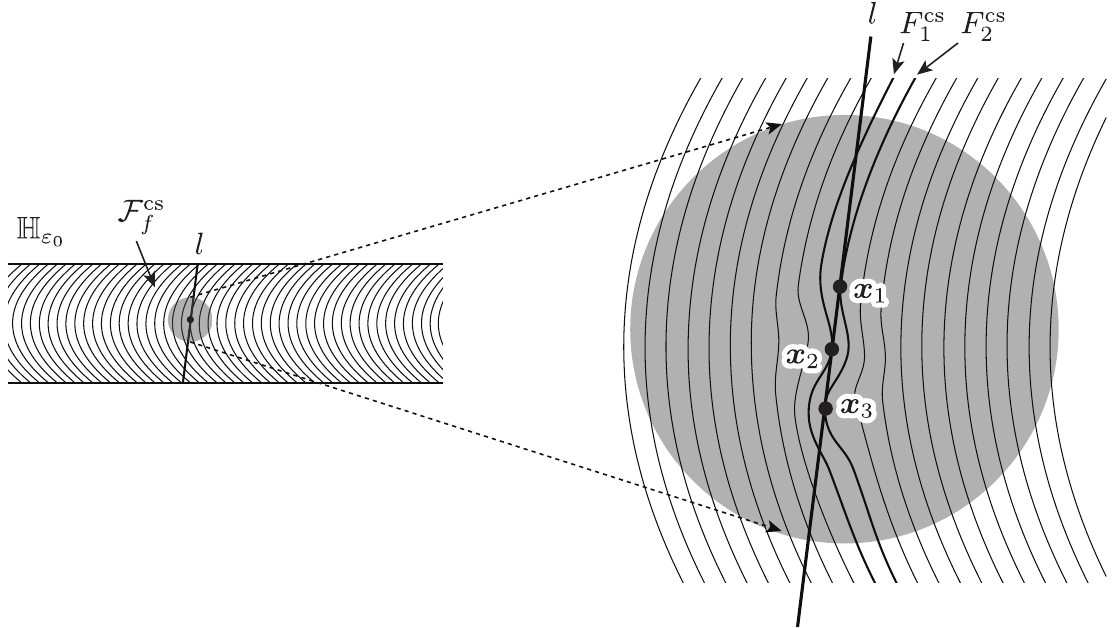}}
\caption{$F_1^{\cs}$ is tangent to $l$ at $\bx_2$ and $F_2^{\cs}$ to $l$ at $\bx_1$, $\bx_3$.
So $\tau(l)\supset \{\bx_1,\bx_2,\bx_3\}$.
In the case of $C^r$ $(r\geq 2)$, since the normal curvature of leaves of $\cf_f^{\cs}$ in the direction of leaves of $\mathcal{L}_{(0;\infty)}$ is positive, they do not have waves that 
$F_1^{\cs}$ or $F_2^{\cs}$ has. }
\label{f_taul}
\end{figure}
Let $\tau(l)$ be the set of such tangency points.
We set 
\[
S^{\tau}=\bigcup\left\{\tau(l)\,;\,\text{$l$ is a leaf of $\mathcal{L}_{(0;\infty)}$}\right\}.
\]
One can also show that, for any leaf $F$ of $\cf_f^{\cs}$ disjoint from the side $I_{\ve_0}^2\times \{-\ve_0,1+\ve_0\}$ of $\bb$, $F\cap S^\tau$ is not empty.
Indeed a point of $F\cap S^\tau$ can be detected by moving leaves of 
$\mathcal{L}_{(0,\infty)}$ from the left side $\{z=-\ve_0\}$ to the right side $\{z=1+\ve_0\}$.

For any binary code $\underline{w}^{(k)}=w_1w_2\ldots w_{k-1}w_k$,
let $S_{\underline{w}^{(k)}}^{\tau}$ be the subset  of $\hh_{\underline{w}^{(k)}}$ defined by 
\[
S_{\underline{w}^{(k)}}^\tau=(f|_{\mathbb{V}_{w_1,f}})^{-1}\circ (f|_{\mathbb{V}_{w_2,f}})^{-1}\circ\cdots \circ (f|_{\mathbb{V}_{w_{k-1},f}})^{-1}\circ (f|_{\mathbb{V}_{w_k,f}})^{-1}(S^\tau).
\]
Since each leaf of $\mathcal{L}_{(0;\infty)}$ is adaptable to $\boldsymbol{C}_{\ve}^{\ru}$, 
from the shape of $\cf_f^{\cs}$ we may assume that $S^\tau$ is contained in $\hh_{\ve_0/2}=
[\frac12-\frac{\ve_0}2,\frac12+\frac{\ve_0}2]\times I_{\ve_0}^2$ 
if necessary replacing $\mathcal{U}_0$ by a smaller neighborhood of $f_0$.

\begin{lem}[cf.\ {\cite[Lemma 8.3]{KNS}}]\label{l_xkinS}
For any positive integer $k_0>0$, there exists a sequence $(\widehat{\boldsymbol{x}}_k)_{k\geq k_0}$ with $\widehat{\boldsymbol{x}}_k\in S_{\widehat{\underline{w}}_k}^\tau$,  
$f^{-2}(\widehat{\boldsymbol{x}}_{k+1})\in \mathbb{U}_{k+1}^{\mathrm{cs}}$ and 
$\|f^{\widehat n_k}(\widehat{\boldsymbol{x}}_k)- f^{-2}(\widehat{\boldsymbol{x}}_{k+1})\|=O(\underline{\lambda}_{\mathrm{u}}^{-(n_0+L(k+1))})$.
\end{lem}
\begin{proof}
Points $\widehat{\boldsymbol{x}}_k$ $(k\geq k_0)$ required in this lemma are defined inductively.
By Lemma \ref{lem-3-1}\,\eqref{lem-3-1(2)}, there exists a leaf $F$ of $\mathcal{F}_f^{\cs}$ in  $\mathbb{U}_{\underline{w}^{(n_0+L(k_0+1))}}^{\mathrm{cs}}$ with 
$F\cap f^{\wh n_{k_0}}(S_{\widehat{\underline{w}}_{k_0}}^\flat)\neq \eset$.
Take a point $\by$ of $F\cap S^\tau$ and set $\wh \bx_{k_0}=f^{-\wh n_{k_0}}(\by)$, 
which is an element of $S_{\wh{\ul\omega}_{k_0}}^\tau$.

We note that
\[
S_{\wh{\ul{w}}_{k+1}}^\tau\subset \hh_{\wh{\ul{w}}_{k+1}}\subset \mathbb{B}^{\mathrm{u}}(\underline{w}^{(n_0+L(k+1))})
\]
for any $k\geq k_0$.
Consider a 1-dimensional $C^0$-foliation $\mathcal{L}_{k+1}$ 
on $\mathbb{B}^{\mathrm{u}}(\underline{w}^{(n_0+L(k+1))})$ such that 
each leaf $l$ is a $C^1$-arc extending a leaf of $\mathcal{L}_{\infty}|_{\hh_{\wh{\ul{w}}_{k+1}}}$ 
and adaptable to $\boldsymbol{C}_\ve^{\ru}$.
Since $\mathbb{U}_{\underline{w}^{(n_0+L(k+1))}}^{\mathrm{cs}}$ is a subset of $f^{-2}(\mathbb{B}^{\mathrm{u}}(\underline{w}^{(n_0+L(k+1))}))$, there exists a leaf 
$l$ of $\mathcal{L}_{k+1}$ such that $f^{-2}(l)$ contains $f^{\wh n_k}(\wh{\bx}_k)$.
See Figure \ref{f_Stau}.
\begin{figure}[hbtp]
\centering
\scalebox{0.6}{\includegraphics[clip]{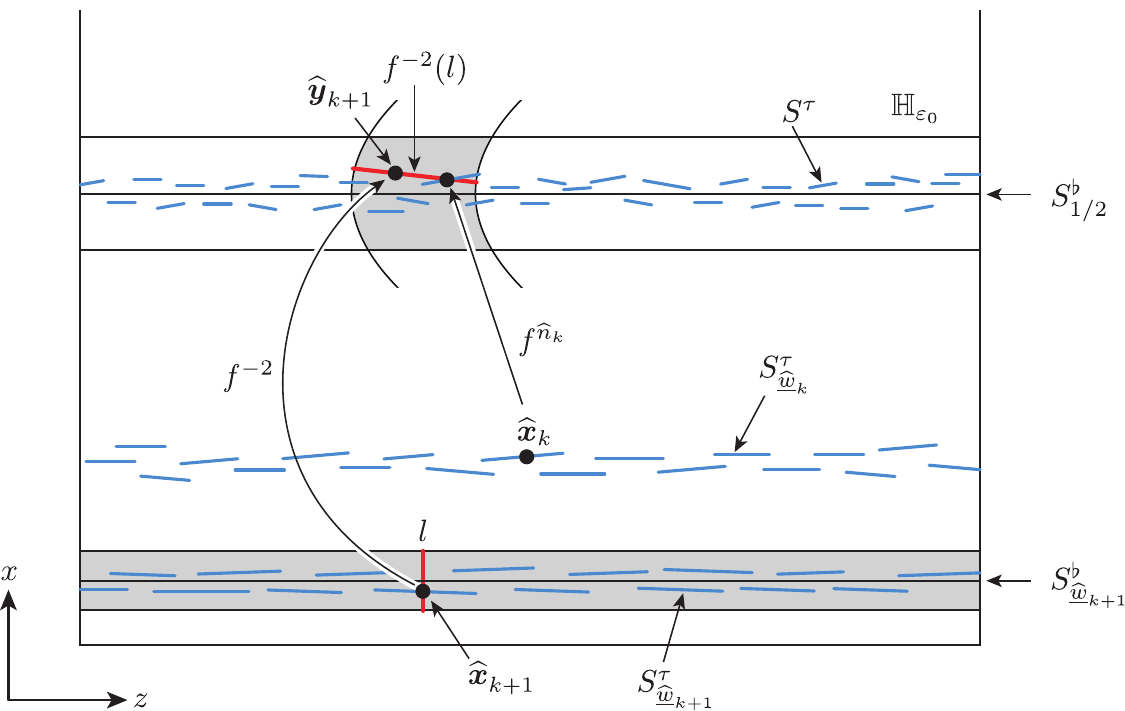}}
\caption{The lower shaded region represents $\mathbb{B}^{\mathrm{u}}(\underline{w}^{(n_0+L(k+1))})$ and the upper does $\mathbb{U}_{\underline{w}^{(n_0+L(k+1))}}^{\mathrm{cs}}$.}
\label{f_Stau}
\end{figure}
Since the leaf $l$ contains a leaf of $\mathcal{L}_{\infty}|_{\hh_{\wh{\ul{w}}_{k+1}}}$ as a sub-arc, 
$l$ has an element of $S_{\wh{\ul w}_{k+1}}^\tau$, which we set $\wh{\bx}_{k+1}$.
Since $f^{n_0+L(k+1)}(l)$ is a $C^1$-arc in $\bb$ compatible with $\boldsymbol{C}_\ve^{\ru}$ and connecting the bottom $\{x=-\ve_0\}$ with the top $\{x=1+\ve_0\}$ of $\bb$, 
it follows that $|f^{n_0+L(k+1)}(l)|\leq (1+\ve)(1+2\ve_0)$.
By \eqref{eqn_lamfkl}, $|l|$ and hence $|f^{-2}(l)|$ are $O(\underline{\lambda}_{\mathrm{u}}^{-(n_0+L(k+1))})$.
This completes the proof.
\end{proof}

Here we write 
\[
f^{-2}(\widehat{\boldsymbol{x}}_{k+1})=\wh\by_{k+1}\quad\text{and}\quad
\bu_k=\wh\by_{k+1}-f^{\widehat n_k}(\widehat{\boldsymbol{x}}_k).
\]
Then Lemma \ref{l_xkinS} implies that 
\begin{equation}\label{eqn_pseudo_orb}
\begin{split}
\bigl(\widehat{\boldsymbol{x}}_{k_0}, f(\widehat{\boldsymbol{x}}_{k_0}),\dots,&f^{\widehat{n}_{k_0}}(\widehat{\boldsymbol{x}}_{k_0}), 
f(\widehat{\boldsymbol{y}}_{k_0+1}),\\
&\widehat{\boldsymbol{x}}_{k_0+1}, f(\widehat{\boldsymbol{x}}_{k_0+1}),\dots,f^{\widehat{n}_{k_0+1}}(\widehat{\boldsymbol{x}}_{k_0+1}),f(\widehat{\boldsymbol{y}}_{k_0+2}),\dots\bigr)
\end{split}
\end{equation}
is an $O(\underline{\lambda}_{\mathrm{u}}^{-(n_0+L(k_0+1))})$-pseudo-orbit of $f$.

\subsection{First perturbations, revisited}\label{ss_first perturb}

First we review the $C^r$-case $(r\geq 2)$ in \cite{KNS} briefly.
Fix a sufficienly large $k_0>0$ and set $f^{\wh n_k}(\wh\bx_k)=(\wh x_k,\wh y_k,\wh z_k)$ and $d_k=\ul{\lambda}_{\ru}^{-(n_0+3k)}$ for $k\geq k_0$.
Let $(\mathbb{D}_k)_{k\geq k_0}$ be the sequence of mutually disjoint cubes in $\mathbb{H}_{\ve_0}$ 
defined as 
\[
\mathbb{D}_k=[\wh x_k-d_k, \wh x_k+d_k]\times 
[\wh y_k-d_k, \wh y_k+d_k]\times [\wh z_k-d_k, \wh z_k+d_k].
\]
We modify $f$ with a series of small perturbations supported on 
$(\mathbb{D}_k)_{k\geq k_0}$.
Then the resulting map $g$ is a $C^r$-diffeomorphism such that 
the $f$-pseudo-orbit \eqref{eqn_pseudo_orb} is an actual orbit of $g$.
The perturbation map $\alpha_k$ on $\mathbb{D}_k$ is a composition of the parallel 
translation along $\bu_k$ and the rotation at $\wh\by_k$ so that 
$\alpha_k(l(f^{\wh n_k}(\wh\bx_k)))$ is tangent to a leaf $F^{\cs}(\wh \by_{k+1})$ of $\cf_f^{\cs}$ at $\wh\by_{k+1}$ 
as illustrated in Figure 8.3 in \cite{KNS}.
Since $\|\bu_k\|=O(\underline{\lambda}_{\mathrm{u}}^{-(n_0+L(k+1))})$, 
$\|\bu_k\|/d_k\leq C\ul{\lambda}_{\ru}^{-(L-3)k}$ for some constant $C>0$  independent of 
$k$. 
As in \cite[Subsection 8.2]{KNS}, one can suppose that $g$ is arbitrarily $C^r$-close to $f$ 
for any fixed $L>9r$.

Now we return to the $C^1$-case.
We fix an integer $L>9$ and define the first perturbation of $f$ by a method similar to the $C^r$-case $(r\geq 2)$.
However the perturbation map $\wh\alpha_k:\mathbb{D}_k\too \bb$ in the $C^1$-case is just the parallel translation along $\bu_k$ and is not composed with any rotation at $\wh\by_{k+1}$, that is, 
$\wh\alpha_k$ is a $C^1$-diffeomorphism on $\mathbb{D}_k$ 
with $\wh\alpha_k(\bx)=\bx+\bu_k$.
This is one of the main differences compared to the $C^r$-case $(r\geq 2)$.
By  employing $C^1$-bump functions $\beta_k$ supported on $\mathbb{D}_k$  as in \cite[Subsection 8.2]{KNS}, we can have the $C^1$-diffeomorphism 
\[
g_*=f\circ \psi:M\too M
\]
arbitrarily $C^1$-close to $f$ if $k_0$ is sufficiently large, 
where $\psi:M\too M$ is the diffeomorphism defined as 
\begin{equation}\label{eqn_psi_n}
\psi(\bx)=\bx+\sum_{k=k_0}^\infty\beta_k(\bx)(\wh\alpha_k(\bx)-\bx)
\end{equation}
on $\bigcup_{k=k_0}^\infty\mathbb{D}_k$ and the identity on $M\setminus \bigcup_{k=k_0}^\infty\mathbb{D}_k$.
Then \eqref{eqn_pseudo_orb} is an actual orbit of $g_*$,

Lemma \ref{lem-3-1}\,\eqref{lem-3-1(0)} implies that  
$f^{\wh n_{l+1}}(\wh\bx_{l+1})$ is located to the left of $f^{\wh n_l}(\wh\bx_l)$
with respect to the leaves of $\cf_f^{\mathrm{cs}}$.
In the $C^r$-case $(r\geq 2)$, the sequence $(\mathbb{D}_k)_{k\geq k_0}$ 
converges a unique tangency point of $f^{-2}(F_{b}^{\rs})$ and a leaf of $\mathcal{L}_{(0,\infty)}$, where $F_{b}^{\rs}$ is the bottom of $\bb^{\ru}(\ul{w}^{(n_0)})$ as shown in Figure \ref{f_Bun0}.
On the other hand, this may not necessarily hold for the $C^1$-case.
We only know that any accumulation point of $(\mathbb{D}_k)_{k\geq k_0}$ in $\hh_{\ve_0}$ is contained in 
$f^{-2}(F_{b}^{\rs})\cap \hh_{\ve_0/2}$. 
\begin{figure}[hbtp]
\centering
\scalebox{0.6}{\includegraphics[clip]{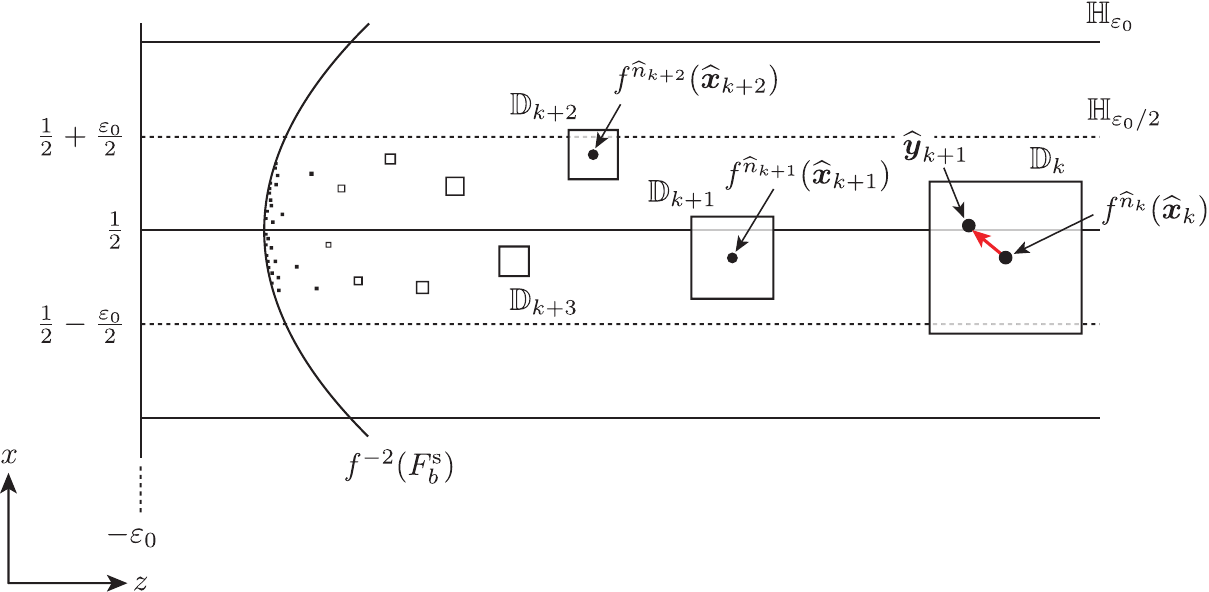}}
\caption{$\widehat\alpha_k(f^{\wh n_k}(\wh\bx_k))=\wh\by_{k+1}$.
The center $f^{\wh n_l}(\wh\bx_l)$ of $\mathbb{D}_l$ belongs to $\mathbb{H}_{\ve_0/2}$.
}
\label{f_Dk}
\end{figure}
However this fact is not required for proving the main theorem, so we will omit the proof.
Since $g_*$ and $f$ differ only in $\bigcup_{k=k_0}^\infty\mathbb{D}_k$, 
the main properties of $f$ still hold for $g_*$, e.g.\ 
\eqref{eqn_BBu}, \eqref{eqn_gg_gamma}, \eqref{eqn_lambda_e/2} and 
$\Lambda_{g_*}=\Lambda_f$, $\cf_{g_*}^{\rs}=\cf_f^{\rs}$.
In contrast, although $\cf_{g_*}^{\cs}$ is not equal to $\cf_f^{\cs}$, we will not use $\cf_{g_*}^{\cs}$ in any case.

\section{Solid cylinders and deviation of foliations}

\subsection{Solid cylinders}\label{ss_solid_sylinder}
We will define solid cylinders $D_k$ in $\bb$ centered at the point $\wh\bx_k$ given in  
Lemma \ref{l_xkinS}.
It is shown in Section \ref{S_second_perturb} that $D_k$ with sufficiently large $k$ is a wandering domain 
for some diffeomorphism $g$ arbitrarily $C^1$-close to $g_*$ and hence to $f$.

Let $\widehat{\underline{w}}_k=\underline{w}^{(n_0+Lk)}\underline{u}_k\underline{\iota}_k\underline{\gamma}^{(m_k)}$ $(k\geq 1)$ be the binary codes 
given in Lemma  \ref{lem-3-1}.
Recall that we have supposed that $|\ul{u}_k|=k^2$ in Subsection \ref{ss_binary_code}.

First we define the sequence $(\xi_k)_{k\geq 1}$ and $(\rho_k)_{k\geq 1}$ by 
\begin{subequations}
\begin{equation}\label{eqn_xi_zeta}
\xi_{k}=\sigma\biggl(\ol{\lambda}_{\rm u}^{\sum_{i=0}^{\infty} \tfrac{\widehat n_{k+i}}{2^{i}}} \biggr)^{-1}
\quad\text{and}\quad
\rho_{k}=\sigma^{-1}\xi_k^{\frac12},
\end{equation}
where $\sigma$ is a positive constant fixed below.
Since
\[
\sum_{i=0}^\infty \frac{\wh n_{k+i}}{2^i}>\wh n_k\sum_{i=0}^\infty \frac1{2^i}=2\wh n_k\quad\text{and}\quad 
\sum_{i=0}^\infty \frac{\wh n_{k+i}}{2^i}=\wh n_k+\frac12\sum_{i=0}^\infty \frac{\wh n_{k+1+i}}{2^i},
\]
it follows that
\begin{equation}\label{eqn_xi_zeta2}
\xi_{k}< \sigma \ol\lambda_{\mathrm{u}}^{\,-2\wh n_{k}}\quad \text{and}
\quad
\xi_{k}=\sigma^{\frac12}\ol\lambda_{\ru}^{\,-\wh n_k}\xi_{k+1}^{\frac12}.
\end{equation}
Thus we also have 
\begin{equation}\label{eqn_xi_zeta3}
\rho_k<\sigma^{-1}\bigl(\sigma\ol\lambda_{\ru}^{\,-2\wh n_k}\bigr)^\frac12
=\sigma^{-\frac12}\ol\lambda_{\ru}^{\,-\wh n_k}
\quad\text{and}\quad 
\frac{\rho_k}{\xi_{k+1}}=\ol\lambda_{\ru}^{\, -2\wh n_k}\xi_k^{-\frac32}.
\end{equation}
\end{subequations}

For any $\wh\bx\in F^{\rs}(\wh\bx_k)$, let $l(\widehat{\boldsymbol{x}})$ be the leaf of $\mathcal{L}_{(\widehat n_k;\infty)}$ containing $\widehat{\boldsymbol{x}}$.
Note that $l(\widehat{\boldsymbol{x}})$ is a maximal $C^1$-arc in $\mathbb{H}_{\widehat{\underline{w}}_k}$, which is divided by $\widehat{\boldsymbol{x}}$ into sub-arcs $l^+(\widehat{\boldsymbol{x}})$, $l^-(\widehat{\boldsymbol{x}})$.
Since $f^{\wh n_k}(\wh \bx)$ and $f^{\wh n_k}(\wh \bx_k)$ are contained 
in the same leaf $F^{\rs}(f^{\wh n_k}(\wh \bx_k))$ of $\mathcal{F}_f^{\rs}|_{\hh_{\ve_0}}$, $f^{\wh n_k}(\wh \bx)$ as well as $f^{\wh n_k}(\wh \bx_k)$ is a point 
in $\hh_{\ve_0/2}$.
So each of $f^{\wh n_k}(l^\pm(\wh \bx))$ contains a component of $f^{\wh n_k}(l(\widehat{\boldsymbol{x}}))
\setminus \hh_{\ve_0/2}$.
This implies that 
$|f^{\wh n_k}(l^\pm(\wh \bx))|\geq \ve_0/2$ 
and hence by \eqref{eqn_lamfkl} 
\begin{equation}\label{eqn_delta_kx}
\delta_{k}(\wh\bx):= \min\bigl\{|l^+(\widehat{\boldsymbol{x}})|, |l^-(\widehat{\boldsymbol{x}})|\bigr\}
\geq 
\frac12 \ve_0\ol\lambda_{\mathrm{u}}^{\,-\widehat n_{k}}
\end{equation}
for any $\wh\bx\in F^{\rs}(\wh\bx_k)$.
We set 
\begin{equation}\label{eqn_delta_k}
\delta(\wh\bx_k)=\delta_k
\end{equation}
for short.

Let $J_k$ be the sub-arc of $l(\widehat{\boldsymbol{x}}_k)$ centered at 
$\widehat{\boldsymbol{x}}_k$ and of length $2\xi_k$.
Let $\bx_k^+$, $\bx_k^-$ be the end points of $J_k$ and $\gg^{\ru}(J_k)$ the closure of 
the component of $\bb\setminus (F^{\rs}(\bx_k^+)\cup F^{\rs}(\bx_k^-))$ containing $\wh\bx_k$.
For $\wh\bx\in F^{\rs}(\wh \bx_k)$, we set $J_k(\wh\bx)=l(\wh\bx)\cap \gg^{\ru}(J_k)$, 
which is divided by $\wh\bx$ into two sub-arcs $J_k^\pm(\wh\bx)$.
In particular $J_k^\pm=J_k^\pm(\wh\bx_k)$.
See Figure \ref{f_GJk}.
\begin{figure}[hbtp]
\centering
\scalebox{0.6}{\includegraphics[clip]{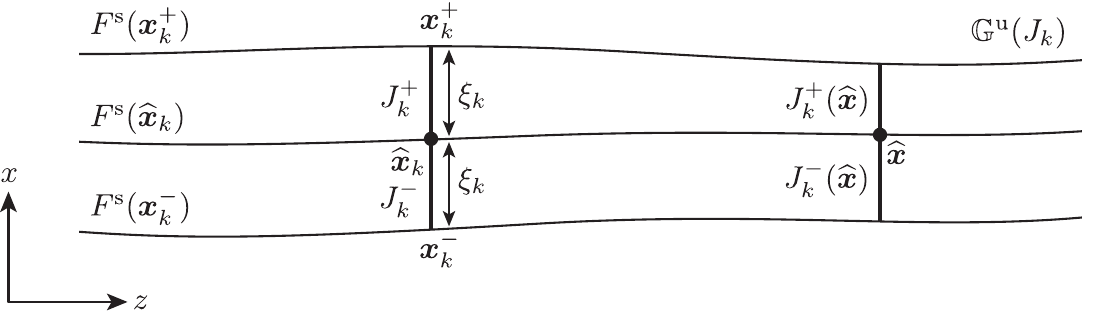}}
\caption{$J_k=J_k^+\cup J_k^-$.}
\label{f_GJk}
\end{figure}
\begin{lem}\label{l_xiJk}
For any $\wh\bx\in F^{\rs}(\wh \bx_k)$,
\[
(1+\ve)^{-1}\xi_k(\ol\lambda_{\ru}^{\,-1}\ul\lambda_{\ru})^{\wh n_k}<|J_k^\pm(\wh\bx)|<
(1+\ve)\xi_k(\ol\lambda_{\ru}\ul\lambda_{\ru}^{\,-1})^{\wh n_k}
\]
holds.
\end{lem}
\begin{proof}
By the property \eqref{F2} of $\cf_f^{\rs}$, 
$F^{\rs}(f^{\wh n_k}(\wh\bx_k))$ and $F^{\rs}(f^{\wh n_k}(\bx_k^\pm))$ are 
parallel to the $yz$-plane.
Since moreover $f^{\wh n_k}(J_k)$ and $f^{\wh n_k}(J_k(\wh\bx))$ 
are compatible with $\boldsymbol{C}_\ve^{\ru}$, 
\[
\frac1{1+\ve}<\frac{|f^{\wh n_k}(J_k^\pm(\wh \bx))|}{|f^{\wh n_k}(J_k^\pm)|}<1+\ve.
\]
By \eqref{eqn_lamfkl} and $|J_k^\pm|=\xi_k$, 
\begin{align*}
&\ol{\lambda}_{\mathrm{u}}^{\,-\wh n_k}|f^{\wh n_k}(J_k^\pm(\wh \bx))|<
|J_k^\pm(\wh \bx)|<\ul{\lambda}_{\mathrm{u}}^{-\wh n_k}|f^{\wh n_k}(J_k^\pm(\wh \bx))|, \\
&\ol{\lambda}_{\mathrm{u}}^{\,-\wh n_k}|f^{\wh n_k}(J_k^\pm)|<
\xi_k<\ul{\lambda}_{\mathrm{u}}^{-\wh n_k}|f^{\wh n_k}(J_k^\pm)|.
\end{align*}
These inequalities complete the proof.
\end{proof}

Let $k_0$ be the positive integer given in Subsection \ref{ss_first perturb}.
By \eqref{eqn_xi_zeta} and \eqref{eqn_xi_zeta2}, $\rho_{k+1}=\sigma^{-1}\xi_{k+1}^{\frac12}=\sigma^{-3/2}\ol\lambda_{\ru}^{\,\wh n_k}\xi_k$.
Since $|J_k|=2\xi_k$, we have by \eqref{eqn_lamDf} 
\[
|g_*^{\wh n_k+2}(J_k)|<c\ol\lambda_{\mathrm{u}}^{\,\wh n_k}\xi_k=c\sigma^{3/2}\rho_{k+1}
\]
for any $k\geq k_0$ and some constant $c>0$ independent of $k$.
So one can choose $\sigma$ so that 
\begin{equation}\label{eqn_gJ}
|g_*^{\wh n_k+2}(J_k)|<\frac13 \rho_{k+1}.
\end{equation}

Let $U_{\rho_k}(\boldsymbol{x})$ be the disk in $F^{\mathrm{s}}(\boldsymbol{x})$ centered at $\boldsymbol{x}$ and of radius $\rho_k$.
We define the solid cylinder $D_k$ $(k\geq k_0)$ in $\mathbb{B}$ by
\begin{equation}\label{eqn_Dk}
D_k=\bigcup_{\boldsymbol{x}\in J_k}U_{\rho_k}(\boldsymbol{x}).
\end{equation}
We suppose that $J_k$ is the core of $D_k$.
See Figure \ref{f_solid_cyl}.
\begin{figure}[hbtp]
\centering
\scalebox{0.6}{\includegraphics[clip]{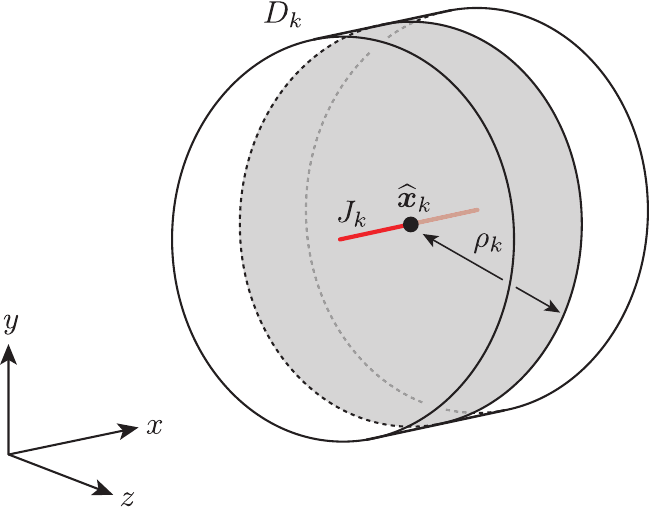}}
\caption{The shaded disk represents $U_{\rho_k}(\wh\bx_k)$.}
\label{f_solid_cyl}
\end{figure}

\subsection{Deviation of foliations}\label{ss_deviation}
Deviation is a new notion not introduced in \cite{KNS}.

For $\bx$, $\by\in \bb$, the \emph{deviation} of $\cf_f^{\rs}$ between 
$\bx$ and $\by$ is the dihedral angle of $D\tau_{\by-\bx}(\bx)(T_{\bx}F^{\rs}(\bx))$ 
and $T_{\by}F^{\rs}(\by)$ at $\by$, where $\tau_{\by-\bx}:\bb\too \rr^3$ is the parallel translation with $\tau_{\by-\bx}(\bx)=\by$.
The deviation is denoted by $\mathrm{dev}_{\cf_f^{\rs}}(\bx,\by)$.
From the definition, we know that the deviation is symmetric, that is, 
$\mathrm{dev}_{\cf_f^{\rs}}(\bx,\by)=\mathrm{dev}_{\cf_f^{\rs}}(\by,\bx)$.
The deviation for the cs-foliation $\cf_f^{\cs}=f^{-2}(\cf_f^{\rs})\cap \hh_{\ve_0}$ and that for the 1-dimensional foliation $\mathcal{L}_{(0,\infty)}$ on $\hh_{\ve_0}$ are defined similarly.

For any $\by\in \hh_{\ve_0}$, 
let $B_d(\by)$ be the closed ball in $\bb$ centered at $\by$ and of radius $d$.
It is well known that the tangent directions of leaves of 
the $C^0$-foliations 
$\cf_f^{\cs}$ and $\cf_f^{\rs}$ vary continuously, for 
example see \cite[Chapter 6]{Sh}.
Since a similar property holds for leaves of the $C^0$-foliation $\mathcal{L}_{(0;\infty)}$, 
we have the following lemma immediately.

\begin{lem}\label{l_e*}
For any $d_0>0$, there exists $\ve_*=\ve_*(d_0)>0$ such that  $\lim\limits_{d_0\to 0}\ve_*=0$ and, for any $\by_0$, $\by_1$ in 
$\hh_{\ve_0}$ with $\|\by_0-\by_1\|<d_0$, 
both $\mathrm{dev}_{\cf_f^{\cs}}(\by_0,\by_1)$ and $\mathrm{dev}_{\mathcal{L}_{(0;\infty)}}(\by_0,\by_1)$ are less than $\ve_*$.
Moreover, for any $\bx_0$, $\bx_1$ in $\bb$ with $\|\bx_0-\bx_1\|<d_0$, 
$\mathrm{dev}_{\cf_f^{\rs}}(\bx_0,\bx_1)$ is also less than $\ve_*$.
\end{lem}

For any $C^1$-arcs $l$, $l'$ (resp.\ $C^1$-surface $F$) in $\bb$ with $l\cap l'\ni \bx$ (resp.\ 
$l\cap F\ni \bx$), 
the angle of $l$ and $l'$ (resp.\ $l$ and $F$) at $\bx$ is denoted by 
$\angle_{\bx}\langle l,l'\rangle$ $\bigl(\text{resp.}\  \angle_{\bx}\langle l,F\rangle\bigr)$.

Since $|J_k|=2\xi_k<2\sigma\ol\lambda_{\ru}^{\,-2\wh n_k}$ by \eqref{eqn_xi_zeta2} for any $k\geq 1$, Lemma \ref{l_xkinS} implies that there exists $d_k>0$ with $\lim\limits_{k\to \infty}d_k=0$ 
and such that $B_{d_k/2}(\wh\by_{k+1})$ contains $f^{\wh n_k}(J_k)\cup g_*^{\wh n_k}(J_k)$.
We set 
\[
\ve_k=\ve_*(d_k).
\]
Then, by Lemma \ref{l_e*}, $\lim\limits_{k\to \infty}\ve_k=0$.
By \eqref{eqn_xi_zeta3} and \eqref{eqn_delta_kx}, we may assume that 
\begin{equation}\label{eqn_delta_rho}
\frac{\delta_k}3>\sqrt{\ve_k}\rho_k 
\end{equation}
holds for any $k\geq k_0$ if necessary replacing $k_0$ by a larger positive integer.

Let $\by_0$ be any point of $g_*^{\wh n_k}(J_k)$ and $\by$ any point of $B_{d_k/2}(\wh\by_{k+1})$.
We write $\by_0'=\psi^{-1}(\by_0)$ for $\psi$ of \eqref{eqn_psi_n}.
Since $\mathcal{L}_\infty$ is $f$-invariant, $f^{\wh n_k}(J_k)$ is a sub-arc of 
$f^{\wh n_k}(l(\wh{\bx}_k))=l(f^{\wh n_k}(\wh{\bx}_k))$.
By Lemma \ref{l_xkinS}, $f^{\wh n_k}(\wh\bx_k)$ is an element of $S^\tau$.
This means that $l(f^{\wh n_k}(\wh \bx_k))$ is tangent to $F^{\cs}(f^{\wh n_k}(\wh \bx_k))$ 
at $f^{\wh n_k}(\wh \bx_k)$.
Since moreover $g_*^{\wh n_k}(J_k)$ is parallel to $f^{\wh n_k}(J_k)$, 
\begin{align*}
\angle_{\by_0}\langle g_*^{\wh n_k}(J_k), \tau_{\by_0-\by}F^{\cs}(\by)\rangle
&=\angle_{\by_0'}\langle f^{\wh n_k}(J_k), \tau_{\by_0'-\by}F^{\cs}(\by)\rangle\\
&\leq \angle_{\by_0'}\langle l(\by_0'), \tau_{\by_0'-f^{\wh n_k}(\wh \bx_k)}
l(f^{\wh n_k}(\wh \bx_k))\rangle\\
&\qquad\qquad+
\angle_{f^{\wh n_k}(\wh \bx_k)}\langle l(f^{\wh n_k}(\wh \bx_k)),\tau_{f^{\wh n_k}(\wh \bx_k)-\by}F^{\cs}(\by)\rangle\\
&
<\ve_k+
\angle_{f^{\wh n_k}(\wh \bx_k)}\langle l(f^{\wh n_k}(\wh \bx_k)),\tau_{f^{\wh n_k}(\wh \bx_k)-\by}F^{\cs}(\by)\rangle\\
&\leq \ve_k+\mathrm{dev}_{\cf_f^{\cs}}(f^{\wh n_k}(\wh \bx_k),\by)<2\ve_k.
\end{align*}
See Figure \ref{f_angle_gF}\,(a).
\begin{figure}[hbtp]
\centering
\scalebox{0.6}{\includegraphics[clip]{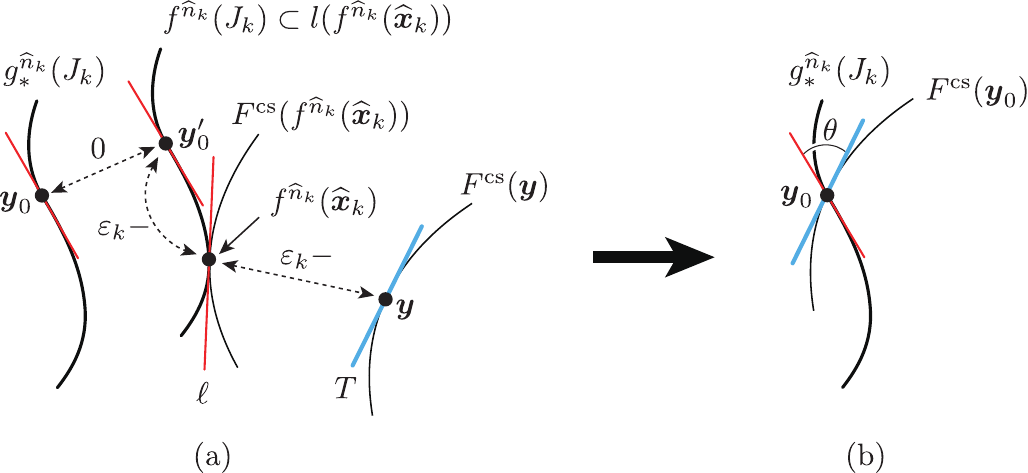}}
\caption{Situation in $\hh_{\ve_0}$.
(a) Each label `$\ve_k-$' means that the difference in direction of tangent lines at the corresponding points is less than $\ve_k$.
$\ell$ is the line tangent to both $f^{\wh n_k}(J_k)$ and $F^{\cs}(f^{\wh n_k}(\wh \bx_k))$ 
at $f^{\wh n_k}(\wh \bx_k)$.
$T$ is the plane tangent to $F^{\cs}(\by)$ at $\by$.
(b) $\theta<2\ve_k$.}
\label{f_angle_gF}
\end{figure}
In particular, in the case when $\by=\by_0$, 
\[
\angle_{\by_0}\langle g_*^{\wh n_k}(J_k), F^{\cs}(\by_0)\rangle<2\ve_k
\]
holds.
See Figure \ref{f_angle_gF}\,(b).
For any $\bx\in  g_*^{\wh n_k+2}(J_k)$, 
we apply the inequality by setting $\by_0=f^{-2}(\bx)$.
Since the differential $Dg_*^2(\bx)$ is a non-singular linear map 
which varies uniformly continuously on $\hh^2_{\ve_0}$, 
there exists 
a constant $c_0>0$ independent of $k$ and satisfying 
\begin{equation}\label{eqn_angle_gJS}
\begin{split}
\angle_{\bx}\langle g_*^{\wh n_k+2}(J_k), \tau_{\bx-\bx_1}F^{\rs}(\wh\bx_{k+1})\rangle
&< \angle_{\bx}\langle g_*^{\wh n_k+2}(J_k), F^{\rs}(\bx)\rangle+
\mathrm{dev}_{\cf_f^{\rs}}(\bx,\bx_1)\\
&<c_0\ve_k+\ve_k=(c_0+1)\ve_k,
\end{split}
\end{equation}
where $\bx_1$ is the point of $F^{\rs}(\wh\bx_{k+1})$ such that $\bx-\bx_1$ is 
parallel to the $x$-axis.
See Figure \ref{f_angle_gFs}.
\begin{figure}[hbtp]
\centering
\scalebox{0.6}{\includegraphics[clip]{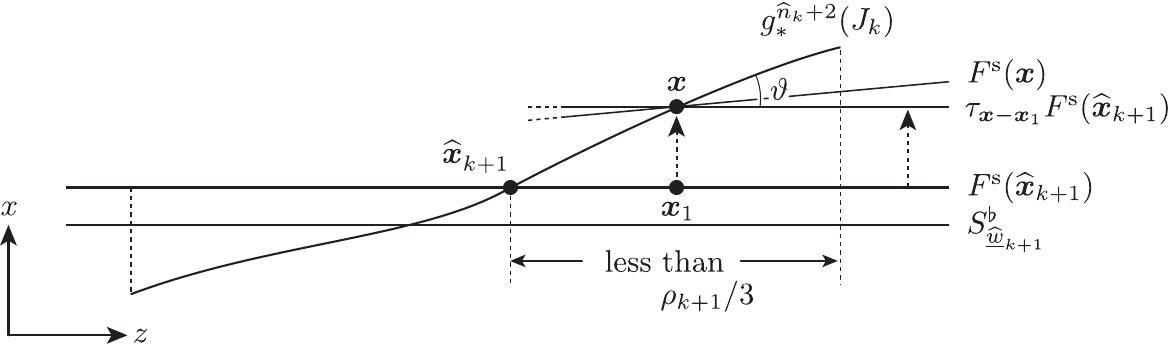}}
\caption{Situation in $\bb^{\ru}(\ul w^{(n_0+L(k+1)})$. $\vartheta<(c_0+1)\ve_k$.}
\label{f_angle_gFs}
\end{figure}
Since $g_*^2=f^2\circ \psi$ and $\psi$ is arbitrarily $C^1$-close to the identity, 
one can take $c_0$ commonly for any $f\in \mathcal{U}_0$.
It follows from \eqref{eqn_gJ} and \eqref{eqn_dist_xFA} in Appendix \ref{S_relative} 
that 
\begin{equation}\label{eqn_dist_xF}
\max_{\,\bx\in g_*^{\wh n_k+2}(J_k)}\mathrm{dist}(\bx, F^{\rs}(\wh\bx_{k+1}))<\frac13 (2c_0+1)\ve_k\rho_{k+1}
\end{equation}
if $\ve$ is sufficiently small and $k_0$ is sufficiently large.

\begin{remark}
Recall that the length of the core $J_{k+1}$ of the solid cylinder $D_{k+1}$ given by \eqref{eqn_Dk} is $2\xi_{k+1}$.
By \eqref{eqn_xi_zeta}, 
\[
\dfrac{\xi_{k+1}}{(2c_0+1)\ve_k\rho_{k+1}}=(2c_0+1)^{-1}\ve_k^{-1}\sigma\sqrt{\xi_{k+1}}.
\]
The sequence $(\sqrt{\xi_{k+1}})_{k\geq k_0}$ decreases exponentially.
On the other hand, from the definition, such rapid decay may not be expected 
for $(\ve_k)_{k\geq k_0}$.
So it would be difficult to exclude the unfavorable case that  
$g_*^{\wh n_k+2}(J_k)$ is not contained in $D_{k+1}$.
See Figure \ref{f_C1Cr}\,(b).
Therefore we need to employ solid cylinders $\wh D_{k+1}$ taller than $D_{k+1}$ and 
containing $g_*^{\wh n_k+2}(J_k)$, see Figure \ref{f_Dhat} below.
\end{remark}

\section{Second perturbations by pressing operations}\label{S_second_perturb}

In this section, we introduce a new perturbation not used in \cite{KNS}.
To help the reader understand the operation, we will first consider a simplified model.

\subsection{\bf A product model}

Suppose that $k\geq k_0$.
Let $\wh E_{k+1}$ be the product solid cylinder $[-\delta_{k+1},\delta_{k+1}]\times U_{\rho_{k+1}}^\flat$ 
and  $\ve_k>0$ a constant satisfying \eqref{eqn_delta_rho}, 
where $U_{\rho_{k+1}}^\flat=\{(y,z)\in \rr^2\,;\, \sqrt{y^2+z^2}\leq \rho_{k+1}\}$.
By using standard arguments on bump functions, one can have  a $C^1$-function 
$\beta_{k+1}:\rr^3\too \rr$ satisfying the following conditions.
\begin{enumerate}[({B}1)]
\makeatletter
\renewcommand{\p@enumi}{B}
\makeatother
\item
$\beta_{k+1}(\bx)=0$ if $\bx\in \rr^3\setminus \mathrm{Int}\wh E_{k+1}$,\label{B1}
\item
$\beta_{k+1}(\bx)\geq 0$ if $\bx\in \wh E_{k+1}$ and $\beta_{k+1}(\bx)=1$ if 
$\bx\in \wh E_{k+1}\setminus \mathcal{N}(\partial \wh E_{k+1})$,\label{B2}
\item
$\left|\dfrac{\partial \beta_{k+1}}{\partial u}(\bx)\right|\leq \dfrac{2}{\sqrt{\ve_k}\rho_{k+1}}$ 
if $\bx\in \wh E_{k+1}$ for $u=x,y,z$,\label{B3}
\end{enumerate}
where $\mathcal{N}(\partial \wh E_{k+1})$ is the $\sqrt{\ve_k}\rho_{k+1}$-neighborhood of $\partial \wh E_{k+1}$ in $\wh E_{k+1}$.

Consider a $C^1$-curve $\ell$ in $\wh E_{k+1}$ parametrized by 
$\ell(t)=(\eta(t),0, t)$ $(-\rho_{k+1}\leq t\leq \rho_{k+1})$, 
where $\eta:[-\rho_{k+1},\rho_{k+1}]\too \rr$ is a $C^1$-function with $\eta(0)=0$ 
and 
\begin{equation}\label{eqn_ae*}
|\eta'(t)|\leq \ve_k
\end{equation}
for any $t$.
Then we have 
\begin{equation}\label{eqn_ae*2}
|\eta(t)|\leq \ve_k|t|
\end{equation}
if $-\rho_{k+1}\leq t\leq \rho_{k+1}$.
Let $\ell_{1/3}$ be the sub-arc of $\ell$ parametrized by $\ell(t)$ with $-\rho_{k+1}/3\leq t\leq \rho_{k+1}/3$.
By \eqref{eqn_delta_rho} and \eqref{eqn_ae*2}, $\ell$ is disjoint from the $\sqrt{\ve_{k}}\rho_{k+1}$-neighborhood of $\{-\delta_{k+1},\delta_{k+1}\}\times U_{\rho_{k+1}}^\flat$ in $\wh E_{k+1}$ 
and $\ell_{1/3}$ is contained in $\wh E_{k+1}\setminus \mathcal{N}(\partial \wh E_{k+1})$, 
see \eqref{eqn_delta_k} for $\delta_{k+1}$.
Let $H_{k+1}:\wh E_{k+1}\too \rr^3$ be the embedding defined as 
$H_{k+1}(x,y,z)=(x-\eta(z),y,z)$.
Then the $C^1$-map $\varPhi_{k+1}:\rr^3\to \rr^3$ is 
defined as
\[
\varPhi_{k+1}=
\begin{cases}
\beta_{k+1}(\bx)(H_{k+1}(\bx)-\bx)+\bx&\text{if}\quad \bx\in \wh E_{k+1},\\
\bx&\text{if}\quad \bx\in \rr^3\setminus \wh E_{k+1}.
\end{cases}
\]
From the definition, we know that 
\[
\varPhi_{k+1}(\ell_{1/3})\subset (\wh E_{k+1}\setminus \mathcal{N}(\partial \wh E_{k+1}))\cap \{(0,y,z)\,|\, (y,z)\in \rr^2\}.
\]
See Figure \ref{f_Ehat}.
\begin{figure}[hbtp]
\centering
\scalebox{0.6}{\includegraphics[clip]{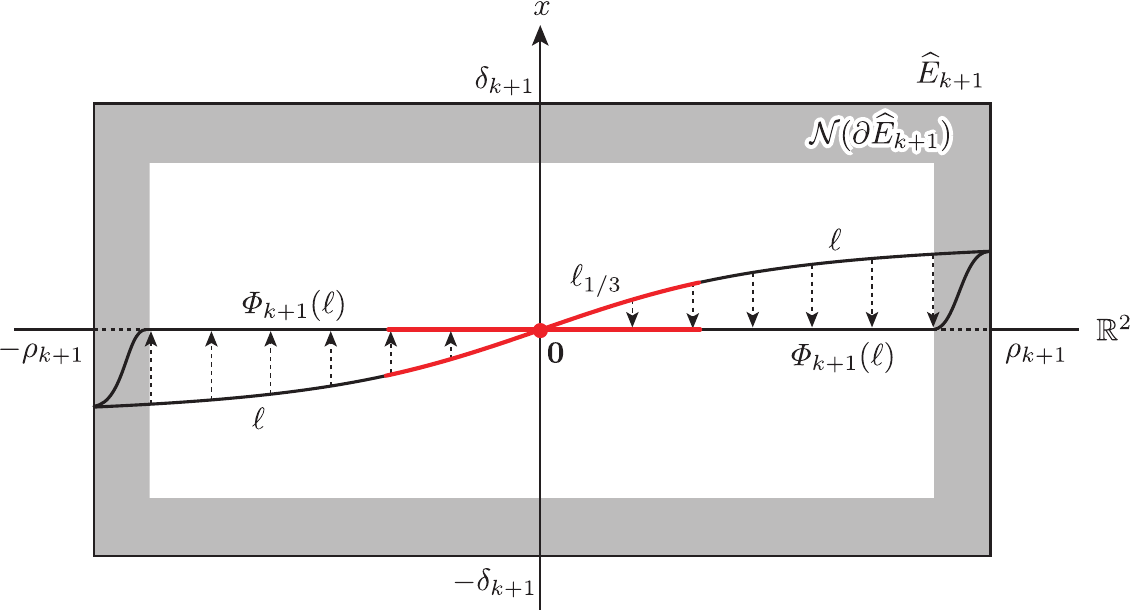}}
\caption{A product model.}
\label{f_Ehat}
\end{figure}
Since $\varPhi_{k+1}(\bx)-\mathrm{id}_{\rr^3}(\bx)=\beta_{k+1}(\bx)(-\eta(z),0,0)$ 
for $\bx\in \wh E_{k+1}$, 
it follows form \eqref{B3}, \eqref{eqn_ae*}, and \eqref{eqn_ae*2} that 
\begin{align*}
\mathrm{dist}_{C^1}\left(\varPhi_{k+1},\mathrm{id}_{\rr^3}\right)
&=\|\beta_{k+1}(-\eta,0,0)\|_\infty\\
&\qquad
+\sum_{u\in \{x,y,z\}}\left\|\frac{\partial \beta_{k+1}}{\partial u}(-\eta,0,0)\right\|_\infty
+\|\beta_{k+1}(-\eta',0,0)\|_\infty\\
&\leq 
\ve_k \rho_{k+1}+3\cdot \frac{2}{\sqrt{\ve_k}\rho_{k+1}}\ve_k\rho_{k+1}+
\ve_k<7\sqrt{\ve_k}
\end{align*}
if $\ve_k(\rho_{k+1}+1)<\sqrt{\ve_k}$.
Since $\mathrm{id}_{\rr^3}$ is a $C^1$-diffeomorphism, $\varPhi_{k+1}$ is also 
a diffeomorphism on $\rr^3$ if $\ve_k$ is sufficiently small.

\subsection{Taller solid cylinders and pressing operation}\label{ss_hatD}

Suppose again that $k\geq k_0$.
We first describe the shape of the curve $g_*^{\wh n_k+2}(J_k)$.
By \eqref{eqn_tang}, $Dg_*^2(\bx)$ is well approximated by 
\[
Df_0^2(\bx)=\begin{pmatrix}
-2a_1\left(x-\frac12\right)&0&a_2\\
0&a_3&0\\
a_4&0&0
\end{pmatrix}
\]
for $\bx=(x,y,z)\in \hh_{\ve_0}$.
Since $g_*^{\wh n_k}(J_k)$ is a $C^1$-arc almost parallel to the $x$-axis in $\hh_{\ve_0}$, 
$g_*^{\wh n_k+2}(J_k)$ is sufficiently $C^1$-close to a straight segment in $\bb$ passing through $\wh\bx_{k+1}$ and parallel to the $z$-axis.

Recall that $l(\widehat{\boldsymbol{x}}_{k+1})$ is the leaf of $\mathcal{L}_\infty|_{\hh_{\widehat{\underline{w}}_{k+1}}}$ containing $\widehat{\boldsymbol{x}}_{k+1}$.
Consider the solid cylinder 
\[
\wh D_{k+1}=\bigcup_{\boldsymbol{y}\in l(\widehat{\boldsymbol{x}}_{k+1})}U_{\rho_{k+1}}(\boldsymbol{y})
\]
in $\mathbb{H}_{\widehat{\underline{w}}_{k+1}}$ centered at $\wh\bx_{k+1}$.
By \eqref{eqn_xi_zeta2} and \eqref{eqn_delta_kx}, we have $\xi_{k+1}<\delta_{k+1}$  
and hence 
\[
|l^\pm(\wh \bx_{k+1})|\geq \delta_{k+1}(\wh \bx_{k+1})=\delta_{k+1}>\xi_{k+1}=|J_{k+1}^\pm|.
\]
This shows that $l(\wh\bx_{k+1})$ contains $J_{k+1}$ and 
hence $D_{k+1}$ is a sub-solid cylinder of $\wh D_{k+1}$.
See Figure \ref{f_Dhat}.
\begin{figure}[hbtp]
\centering
\scalebox{0.6}{\includegraphics[clip]{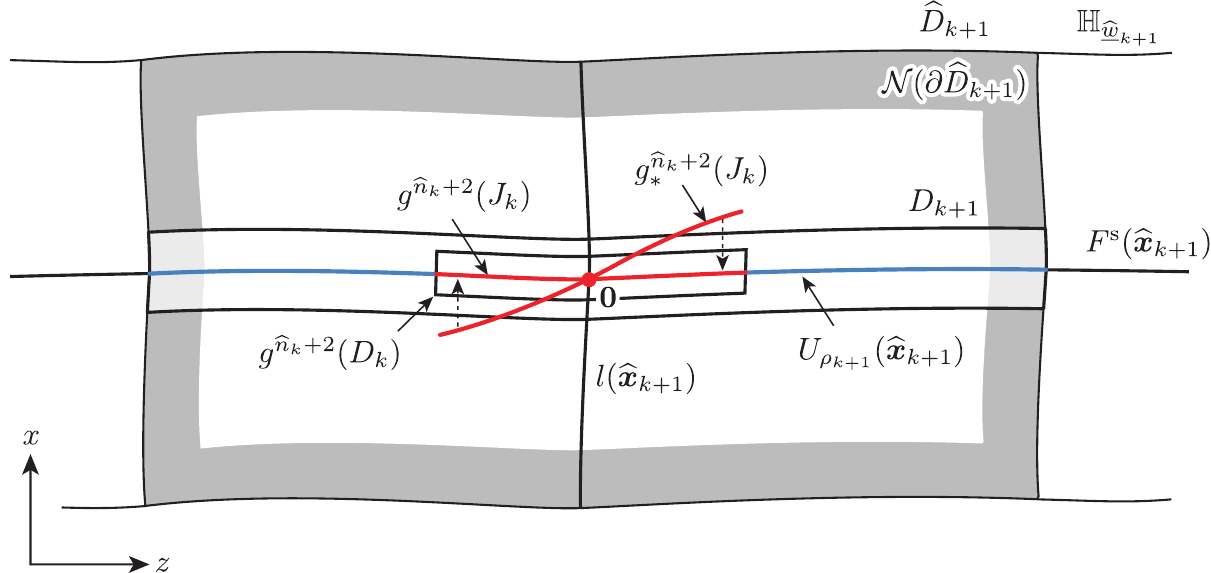}}
\caption{A general model with $\wh\bx_{k+1}=\mathbf{0}$.
}
\label{f_Dhat}
\end{figure}

To simplify representation, we temporarily use the coordinate on $\bb$ with $\wh \bx_{k+1}=(0,0,0)$ obtained by the coordinate change $\bx\longmapsto \bx-\wh\bx_{k+1}$ 
on $\bb$.
By the shape of $g_*^{\wh n_k+2}(J_k)$ and \eqref{eqn_gJ}, one can have a parametrization $(\gamma(t),\iota(t),t)$ $(-t_0\leq t\leq t_1)$ of $g_*^{\wh n_k+2}(J_k)$ with $(\gamma(0),\iota(0),0)=(0,0,0)$ 
for some $0<t_0,t_1<\rho_{k+1}/3$.
Let $\gamma_0:[-t_0, t_1]\too \rr$ be a $C^1$-function  such that 
$(\gamma_0(t),\iota(t),t)$ is contained in $U_{\rho_{k+1}}(\wh\bx_{k+1})$ for any 
$-t_0\leq t\leq t_1$.
Intuitively the curve is obtained by 
the $x$-directional orthogonal projection of $g^{\wh n_k+2}_*(J_k)$ into $U_{\rho_{k+1}}(\wh\bx_{k+1})$.
Let $\wh\gamma:[-t_0, t_1]\too \rr$ be the $C^1$-function 
defined as $\wh\gamma(t)=\gamma(t)-\gamma_0(t)$.
Since $F^{\rs}(\wh\bx_{k+1})$ is compatible with $\boldsymbol{C}_\ve^{\cs}$, 
we have by \eqref{eqn_angle_gJS} and \eqref{eqn_gammaA} 
\[
|\wh\gamma'(t)|\leq (2c_0+1)\ve_k
\]
if $\ve$ is sufficiently small and $k_0$ is sufficiently large.
Fix a $C^1$-function $\wh\eta:\rr\too \rr$ extending 
$\wh\gamma$ and satisfying 
$
|\wh\eta'(t)|\leq (2c_0+1)\ve_k
$
for any $t\in \rr$.
The condition on $\wh\eta'(t)$ corresponds to \eqref{eqn_ae*} in the product model.
The $C^1$-embedding $\wh H_{k+1}:\wh D_{k+1}\too \bb$ is 
defined as $\wh H_{k+1}(x,y,z)=(x-\wh\eta(z),y,z)$.

Let $\mathcal{N}(\partial \wh D_{k+1})$ be the $\sqrt{\ve_{k}}\rho_{k+1}$-neighborhood of $\partial \wh D_{k+1}$ in $\wh D_{k+1}$.
By \eqref{eqn_xi_zeta3}, \eqref{eqn_delta_kx} and \eqref{eqn_dist_xF}, one can suppose 
that $g_*^{\wh n_k+2}(J_k)$ is contained in $\wh D_{k+1}\setminus \mathcal{N}(\partial \wh D_{k+1})$ for any $k\geq k_0$ if necessary replacing $k_0$ with a larger positive integer.
See Figure \ref{f_Dhat} again.

By considering a bump function as \eqref{B1}--\eqref{B3} supported on $\wh D_{k+1}$, 
we have a $C^1$-diffeomorphism 
\begin{equation}\label{eqn_Psik+1}
\varPsi_{k+1}:M\too M
\end{equation}
satisfying the following conditions, where we use $\mathcal{N}(\partial \wh D_{k+1})$ instead of $\mathcal{N}(\partial \wh E_{k+1})$ in the product model.
\begin{itemize}
\setlength{\itemindent}{-16pt}
\item
$\mathrm{dist}_{C^1}(\varPsi_{k+1},\mathrm{id}_M)<7(2c_0+1)\sqrt{\ve_k}$,
\item
$\varPsi_{k+1}|_{M\setminus \wh D_{k+1}}=\mathrm{id}_{M\setminus \wh D_{k+1}}$,
\item
$\varPsi_{k+1}(g_*^{\wh n_k+2}(\wh\bx_k))=\wh\bx_{k+1}$,\quad  
$\varPsi_{k+1}(g_*^{\wh n_k+2}(J_k))\subset U_{\rho_{k+1}}(\wh\bx_{k+1})$.
\end{itemize}

\subsection{Proof of Theorem \ref{mainthm}}\label{ss_proof_thmA}

Take any element $\bx$ of $\Lambda_{f}^{(\mathrm{mj})}$ and 
set $\mathcal{I}_{f}(\bx)=(v_j)_{j\in \mathbb{Z}}$.
We  set $\alpha_1=n_0+L$ and define inductively a strictly increasing sequence $(\alpha _k)_{k\geq 1}$ of 
positive integers satisfying \eqref{D1} of Definition \ref{describable} with $\beta_k=k^2$.
Suppose that $\alpha_1,\dots,\alpha_k$  are already determined.
Then the binary code $\underline u_k$ is defined by
\begin{equation}\label{eqn_u_k}
\underline u_k=(v_{\alpha_k+1}v_{\alpha_k+2}\dots v_{\alpha_k+k^2}).
\end{equation}
By applying Lemma \ref{lem-3-1} to $\underline u_k$, 
we have finite binary codes $\ul\iota_k$ and $\ul\gamma^{(m_k)}$ 
satisfying the conditions of the lemma.
Then $\alpha_{k+1}$ is defined by 
\begin{equation}\label{eqn_alphak+1}
\alpha_{k+1}=\alpha_k+k^2+|\underline{\iota}_k|+m_k+2+n_0+L(k+1).
\end{equation}
See the thick black line segment in Figure \ref{f_wivi} below.

Let $\underline{\widehat w}_k=\underline{w}^{(n_0+Lk)}\underline{u}_k\underline{\iota}_k\underline{\gamma}^{(m_k)}$ be the binary code defined in Lemma \ref{lem-3-1}.
Suppose that $\wh\bx_k\in S_{\wh{\ul{w}}_k}^\tau\subset \bb(\wh{\ul{w}}_k)$ is  the point constructed as in Lemma \ref{l_xkinS} such that the sequence $(\wh\bx_k)_{k\geq k_0}$ determines the $C^1$-diffeomorphism $g_*:M\too M$ arbitrarily $C^1$-close to $f$, 
see Subsection \ref{ss_first perturb}.
Let $\Lambda_{g_*}^{(\mathrm{mj})}$ be the continuation of $\Lambda_{f}^{(\mathrm{mj})}$ 
and $\bx_{g_*}$ the element of $\Lambda_{g_*}^{(\mathrm{mj})}$ corresponding to $\bx$.
For any $k_0\leq j<a$, consider the composition 
$\wt \varPsi_{j+1,a}=\varPsi_a\circ\cdots\circ \varPsi_{j+1}:M\too M$.
From our construction, one can suppose that $\ve_k>0$ is arbitrarily small for any $k\geq k_0$. 
For $k_0\leq j<a<a'$, $\wt\varPsi_{j+1,a}(\bx)\neq \wt\varPsi_{j+1,a'}(\bx)$ only  if 
$\bx\in \bigcup_{\,l=a+1}^{\,a'}\wh D_l$.
It follows that
\[
\mathrm{dist}_{C^1}(\wt \varPsi_{j+1,a},\wt \varPsi_{j+1,a'})\leq 7(2c_0+1)\sqrt{\ve_a}.\]
Hence $(\wt\varPsi_{j+1,a})_{a=j+1}^\infty$ is a Cauchy sequence 
in the complete metric space $(\mathrm{Map}^1(M),\|\cdot\|_{C^1})$ of $C^1$-maps on $M$.
Thus the sequence $C^1$-converges to a $C^1$-map  
$\wt\varPsi_{j+1,\infty}:M\too M$ as $a\to \infty$.
Similarly, for $k_0\leq j<j'$, $\wt\varPsi_{j+1,\infty}(\bx)\neq \wt\varPsi_{j'+1,\infty}(\bx)$ only  if 
$\bx\in \bigcup_{\,m=j}^{\,j'-1}\wh D_m$.
So the sequence $(\wt\varPsi_{j+1,\infty})_{j\geq k_0}$ $C^1$-converges to $\mathrm{id}_M$ as $j\to \infty$.
Since $\mathrm{id}_M$ is a diffeomorphism, $\wt\varPsi_{j+1,\infty}$ is also a $C^1$-diffeomorphism for any sufficiently large $j$.
So there exists $k_1\geq k_0$ such that 
\begin{equation}\label{eqn_def_g}
g=\wt\varPsi_{k_1+1,\infty}\circ g_*:M\too M
\end{equation}
is a $C^1$-diffeomorphism arbitrarily $C^1$-close to $g_*$ and hence to $f$.

In the remainder of this paper, we prove that $g$ is a diffeomorphism 
desired in Theorem \ref{mainthm}.

We represent $\widehat{\underline{w}}_k=\underline{w}^{(n_0+Lk)}\underline{u}_k\underline{\iota}_k\underline{\gamma}^{(m_k)}$ 
as $(w_1w_2\dots w_{\widehat n_k})$.
The condition \eqref{eqn_u_k} implies  
\begin{equation}\label{eqn_wiva}
w_{n_0+Lk+i}=v_{\alpha_k+i}\quad(i=1,\dots,k^2).
\end{equation}
See Figure \ref{f_wivi}.
\begin{figure}[hbtp]
\centering
\scalebox{0.6}{\includegraphics[clip]{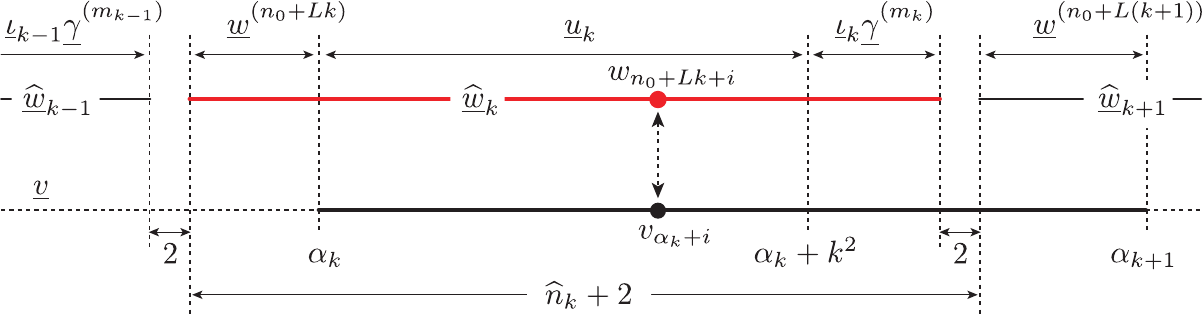}}
\caption{Recall that $|\widehat{\underline{w}}_k|=\widehat n_k$ and $g_*^{\wh n_k+2}(\wh\bx_k)=\wh \bx_{k+1}$.}
\label{f_wivi}
\end{figure}
We set $\wh\bx_1=g_*^{n_0+Lk_0-\alpha_{k_0}}(\wh\bx_{k_0})$.
Then $g_*^{\alpha_k}(\wh \bx_1)=g_*^{n_0+Lk
}(\wh \bx_k)$ holds for any $k\geq k_0$ and hence $g_*^{\alpha_k+i}(\wh \bx_1)=g_*^{n_0+Lk+i}(\wh \bx_k)$ for $i\geq 0$.
So \eqref{eqn_wiva} implies that both $g_*^{\alpha_k+i}(\bx_{g_*})$ and $g_*^{\alpha_k+i}(\wh \bx_1)$ 
are contained in $g_*^{i-1}(\bb^{\ru}(\underline{u}_k\underline{\iota}_k\underline{\gamma}^{(m_k)}))$ for $i=1,\dots,k^2$ if $k\geq k_0$.
For any $\ve>0$, we will show that there exists an integer $k_2\geq k_1$ such that 
\begin{equation}\label{eqn_gxgx}
\|g_*^{\alpha_k+i}(\wh \bx_1)-g_*^{\alpha_k+i}(\bx_{g_*})\|<\ve
\end{equation}
holds for any $k\geq k_2$ and $i$ with $k\leq i\leq k^2-k$.
In fact, by \eqref{eqn_lambda_e/2}, 
there exists a constant $c_1>0$ independent of $k$ such that, 
for any leaf $F^{\rs}$ of $\cf_f^{\rs}$ with $g_*^{i}(\bb^{\ru}(\underline{u}_k\underline{\iota}_k\underline{\gamma}^{(m_k)}))\cap F^{\rs}\neq \emptyset$, the diameter of the section $\sigma_i(F^{\rs})$ of $g_*^{i-1}(\bb^{\ru}(\underline{u}_k\underline{\iota}_k\underline{\gamma}^{(m_k)}))$ by $F^{\rs}$ 
is less than 
$c_1\ol\lambda_{\cs 1}^{\,k}$ for any $k\leq i\leq k^2$.
Moreover, by \eqref{eqn_lamfkl}, 
the height $h_i$ of $\bb^{\ru}(\underline{u}_{k(i)}\underline{\iota}_k\underline{\gamma}^{(m_k)})$ is less than $c_1\ul{\lambda}_{\ru}^{-(k+|\iota_k|+m_k)}$ for any $1\leq i\leq k^2-k$, where 
$\underline u_{k(i)}=(v_{\alpha_k+i}v_{\alpha_k+1+i}\dots v_{\alpha_k+k^2})$ and 
\[
h_i=\max\left\{\,|l|\,;\, \text{$l$ is a maximal $C^1$-arc in $\bb(\underline{u}_{k(i)}\underline{\iota}_k\underline{\gamma}^{(m_k)})$ compatible with $\boldsymbol{C}_\ve^{\ru}$}\,\right\}.
\]
See Figure \ref{f_height}.
\begin{figure}[hbtp]
\centering
\scalebox{0.6}{\includegraphics[clip]{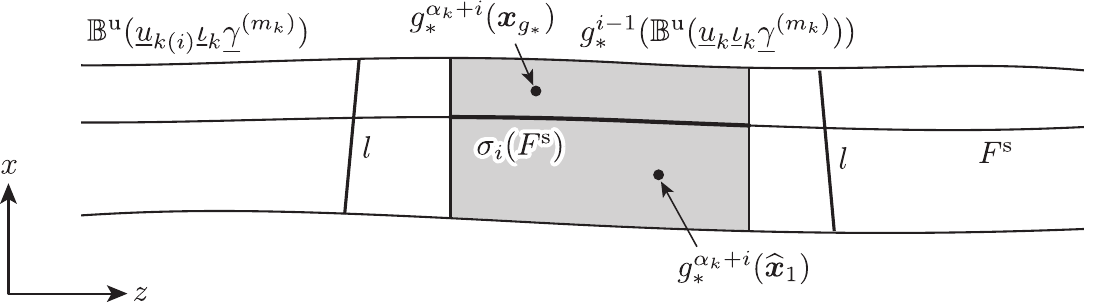}}
\caption{The diameter of the section $\sigma_i(F^{\rs})$ decreases, whereas the height of $\bb^{\ru}(\underline{u}_{k(i)}\underline{\iota}_k\underline{\gamma}^{(m_k)})$ increases  as $i\to k^2$.
}
\label{f_height}
\end{figure}
Thus \eqref{eqn_gxgx} holds if we take $k_2$ sufficiently large.
By \eqref{eqn_alphak+1}, we have $\alpha_{k+1}-\alpha_k=k^2+O(k)$ and hence
\[
\lim_{k\to\infty}\frac{k^2-k}{\alpha_{k+1}-\alpha_k}=1.
\]
This means that almost all $i$ from $1$ to $\alpha_{k+1}-\alpha_k$ satisfy \eqref{eqn_gxgx} 
for any sufficiently large $k$.
It follows that 
\begin{equation}\label{eqn_gg}
\lim _{n\to \infty} 
\frac{1}{n}\sum_{i=0}^{n-1}
\mathrm{dist}\,(g_*^i(\wh\bx_1),g_*^i(\bx_{g_*}))=0.
\end{equation}

\bigskip

Since $\wh n_k=k^2+O(k)$, 
there exists a positive integer $k_3\geq k_2$, for any $\eta$ with $0<\eta<1$, such that 
$\widehat n_{k+i}<(1+\eta)^i\widehat n_{k}$ if $k\geq k_3$ and $i\geq 1$.
Then we have 
\begin{equation}\label{eqn_3/2}
\frac{3}{2}\sum_{i=0}^{\infty}\frac{\widehat n_{k+i}}{2^{i}}
\leq 
\frac{3\widehat n_{k}}{2}\sum_{i=0}^{\infty}
\biggl(\frac{1+\eta}{2}\biggr)^{i}
=\frac{3\widehat n_{k}}{1-\eta}=(3+\eta_{1})\widehat n_{k}
\end{equation}
if $k\geq k_3$, where $\eta_{1}=3\eta/(1-\eta)$.

We denote the total number of $0$ and $1$ entries in the code $\wh{\ul\omega}_k$ by 
$\wh n_{k(0)}$ and $\wh n_{k(1)}$ respectively.
In particular, 
\[
\wh n_k=|\wh{\ul\omega}_k|=\wh n_{k(0)}+\wh n_{k(1)}.
\]
As shown in \cite[Subsection 9.1]{KNS}, 
the majority condition \eqref{eqn_majority} on $\bx_{g_*}$ 
implies the inequality
\begin{equation}\label{eqn_nk1_nk2}
\wh n_{k(1)}<(1+\eta_0)\wh n_{k(0)}
\end{equation}
for all sufficiently large $k$, where $\eta_0=4\eta/(1-\eta)$.

\medskip

To prove Theorem \ref{mainthm}, we need the following refined version of \cite[Lemma 9.3]{KNS}.

\begin{lem}\label{lem7.3}
$$
\ol\lambda_{\rm cs0}^{\,\widehat n_{k(0)}}\ol\lambda_{\rm cs1}^{\,\widehat n_{k(1)}}\rho_{k}=o(\xi_{k+1}(\ol\lambda_{\ru}^{\,-1}\ul\lambda_{\ru})^{\wh n_{k+1}}).
$$
\end{lem}
\begin{proof}
Since 
\[
(\ol\lambda_{\ru}\ul\lambda_{\ru}^{-1})^{\wh n_{k+1}}
<(\ol\lambda_{\ru}\ul\lambda_{\ru}^{-1})^{(1+\eta)\wh n_{k}}
<(\ol\lambda_{\ru}\ul\lambda_{\ru}^{-1})^{(1+\eta_1)\wh n_{k}},
\]
we have 
by \eqref{eqn_xi_zeta3}, \eqref{eqn_3/2} and \eqref{eqn_nk1_nk2} 
\begin{align}\label{eqn_lambda_lambda}
\frac{ \ol\lambda_{\rm cs0}^{\,\widehat n_{k(0)}}\ol\lambda_{\rm cs1}^{\,\widehat n_{k(1)}}\rho_{k}}{\xi_{k+1}(\ol\lambda_{\ru}^{\,-1}\ul\lambda_{\ru})^{\wh n_{k+1}}}
&<\ol\lambda_{\rm cs0}^{\,\widehat n_{k(0)}}\ol\lambda_{\rm cs1}^{\,\widehat n_{k(1)}} \left(\ol{\lambda}_{\rm u}^{\,-2\wh n_k}\xi_k^{-\frac32}\right)(\ol\lambda_{\ru}\ul\lambda_{\ru}^{-1})^{(1+\eta_1)\wh n_k}\nonumber\\
&=
\sigma^{-\frac{3}{2}} \ol\lambda_{\rm cs0}^{\,\widehat n_{k(0)}}\ol\lambda_{\rm cs1}^{\,\widehat n_{k(1)}} \ol{\lambda}_{\rm u}^{\,-2 \widehat n_{k}}
\biggl(\ol{\lambda}_{\rm u}^
{\,\sum_{i=0}^{\infty}\tfrac{\widehat n_{k+i}}{2^{i}}}
\biggr)^{\frac{3}{2}}(\ol\lambda_{\ru}\ul\lambda_{\ru}^{-1})^{(1+\eta_1)\wh n_k}\nonumber\\
&\leq \sigma^{-\frac{3}{2}} \ol\lambda_{\rm cs0}^{\,\widehat n_{k(0)}}\ol\lambda_{\rm cs1}^{\,\widehat n_{k(1)}} \ol{\lambda}_{\rm u}^{\,(2+2\eta_1)\widehat n_{k}}\ul\lambda_{\ru}^{-(1+\eta_1)\wh n_k}.
\end{align}
Since 
$\ol\lambda_{\rm cs0}\ol\lambda_{\rm cs1}\ol\lambda_{\ru}^{\,4}\ul\lambda_{\ru}^{-2}<1$ 
by \eqref{eqn_eigen_v2}, it follows that  
\[
\ol\lambda_{\rm cs0}\ol\lambda_{\rm cs1}^{\,(1+\eta_0)}\ol\lambda_{\mathrm{u}}^{\,(2+\eta_0)(2+2\eta_{1})}\ul\lambda_{\ru}^{-(2+\eta_0)(1+\eta_1)}<1
\]
if $\eta>0$ is sufficiently small.
On the other hand, since 
\[
\ol\lambda_{\rm cs1}\ol\lambda_{\mathrm{u}}^{\,2}\ul\lambda_{\ru}^{-1}
=\ol\lambda_{\rm cs1}\ol\lambda_{\mathrm{u}}(\ol\lambda_{\mathrm{u}}\ul\lambda_{\ru}^{-1})>\ol\lambda_{\rm cs1}\ol\lambda_{\mathrm{u}}>2\cdot\frac12=1
\]
by \eqref{eqn_eigen_v2}, 
one can also suppose that 
$
\ol\lambda_{\rm cs1}\ol\lambda_{\mathrm{u}}^{\,(2+2\eta_{1})}\ul\lambda_{\ru}^{-(1+\eta_1)}>1$.
Since $\wh n_k=\wh n_{k(0)}+\wh n_{k(1)}$, we have by \eqref{eqn_lambda_lambda} 
\begin{align*}
\frac{\ol\lambda_{\rm cs0}^{\,\widehat n_{k(0)}}\ol\lambda_{\rm cs1}^{\,\widehat n_{k(1)}}\rho_{k}}{\xi_{k+1}(\ol\lambda_{\ru}^{\,-1}\ul\lambda_{\ru})^{\wh n_{k+1}}}
&
\leq \sigma^{-\frac{3}{2}} \ol\lambda_{\rm cs0}^{\,\widehat n_{k(0)}}\ol\lambda_{\rm cs1}^{\,\widehat n_{k(1)}} \ol{\lambda}_{\rm u}^{\,(2+2\eta_1)(\widehat n_{k(0)}+\widehat n_{k(1)})}\ul\lambda_{\ru}^{-(1+\eta_1)(\widehat n_{k(0)}+\widehat n_{k(1)})}\\
&= 
\sigma^{-\frac32}
(\ol\lambda_{\rm cs0}\ol\lambda_{\mathrm{u}}^{\,(2+\eta_{1})}\ul\lambda_{\ru}^{-(1+\eta_1)})^{\widehat n_{k(0)}}(\ol\lambda_{\rm cs1}\ol\lambda_{\mathrm{u}}^{\,(2+\eta_{1})}\ul\lambda_{\ru}^{-(1+\eta_1)})^{\widehat n_{k(1)}}\\
&<
\sigma^{-\frac32}
(\ol\lambda_{\rm cs0}\ol\lambda_{\mathrm{u}}^{\,(2+2\eta_{1})}\ul\lambda_{\ru}^{-(1+\eta_1)})^{\widehat n_{k(0)}}(\ol\lambda_{\rm cs1}\ol\lambda_{\mathrm{u}}^{\,(2+\eta_{1})}\ul\lambda_{\ru}^{-(1+\eta_1)})^{(1+\eta_0)\widehat n_{k(0)}}\\
&\leq 
\sigma^{-\frac32}
\bigl(\ol\lambda_{\rm cs0}\ol\lambda_{\rm cs1}^{\,(1+\eta_0)}\ol\lambda_{\mathrm{u}}^{\,(2+\eta_0)(2+2\eta_{1})}\ul\lambda_{\ru}^{-(2+\eta_0)(1+\eta_1)}\bigr)^{\widehat n_{k(0)}}\rightarrow 0
\end{align*}
as $k\rightarrow \infty$.
This completes the proof.
\end{proof}

Now we are ready to prove Theorem \ref{mainthm}.

\begin{proof}[Proof of Theorem \ref{mainthm}]
By \eqref{eqn_gJ} and \eqref{eqn_def_g}, for any $k\geq k_1$, 
\begin{equation}\label{eqn_gJkU}
g^{\wh n_k+2}(J_k)=\wt\varPsi_{k_1+1,\infty}(g_*^{\wh n_k+2}(J_k))=
\varPsi_{k+1}(g_*^{\wh n_k+2}(J_k))\subset  U_{\rho_{k+1}}(\wh\bx_{k+1}) .
\end{equation}
Moreover, by \eqref{eqn_lambda_e/2} and Lemma \ref{lem7.3}, we have  
\begin{equation}\label{eqn_diam_gU}
\mathrm{diam}\,g_*^{\wh n_k+2}(U_{\rho_k}(\bx))
\leq 2\ol\lambda_{\rm cs0}^{\,\widehat n_{k(0)}}\ol\lambda_{\rm cs1}^{\,\widehat n_{k(1)}}\rho_{k}
=o(\xi_{k+1}(\ol\lambda_{\ru}^{\,-1}\ul\lambda_{\ru})^{\wh n_{k+1}})
\end{equation}
for any $\bx\in J_k$.
Hence, by Lemma \ref{l_xiJk} together with \eqref{eqn_gJkU} and \eqref{eqn_diam_gU}, 
one can  have an integer $j_0\geq k_1$ such that, for any $k\geq j_0$,  
\[
g^{\,\wh n_k+2}(D_k)=\wt\varPsi_{k_1+1,\infty}(g_*^{\wh n_k+2}(D_k))
=\varPsi_{k+1}(g_*^{\wh n_k+2}(D_k))\subset D_{k+1}.
\]
See Figure \ref{f_Dhat} again.
Thus $\mathrm{Int}D_{j_0}$ is a wandering domain of $g$ 
with $\mathrm{Int}D_{j_0}\ni \wh\bx_{j_0}$ and 
\begin{equation}\label{eqn_diamD}
\lim_{i\to \infty}\mathrm{diam}\,g^i(\mathrm{Int}D_{j_0})=0.
\end{equation}
So $D_g=g^{n_0+Lj_0-\alpha_{j_0}}(\mathrm{Int}D_{j_0})$ is also a wandering domain of $g$.

We set $\by_1=g^{n_0+Lk_0-\alpha_{k_0}}(\wh\bx_{k_0})$.
Since $\bx_{k+1}=g^{\wh n_k+2}(\wh\bx_k)=g^{\alpha_{k+1}-\alpha_k-L}(\wh\bx_k)$  (see Figure \ref{f_wivi}) for any $k\geq k_0$, 
\[
\wh\bx_{j_0}
=g^{\alpha_{j_0}-\alpha_{k_0}-L(j_0-k_0)}(\wh \bx_{k_0})
=g^{\alpha_{j_0}-(n_0+Lj_0)}(\wh \by_1).
\]
Since $\wh\bx_{j_0}\in \mathrm{Int} D_{j_0}$, this shows that 
\begin{equation}\label{eqn_byinD}
\wh \by_1\in D_g.
\end{equation}
From the third condition for $\varPsi_{k+1}$ of \eqref{eqn_Psik+1}, we have $g^i(\wh\bx_{k_0})= g_*^i(\wh\bx_{k_0})$ for any $i\geq 0$.
Since moreover $g^{\alpha_{k_0}-(n_0+Lk_0)}(\wh\by_1)=g_*^{\alpha_{k_0}-(n_0+Lk_0)}(\wh\bx_1)=\wh\bx_{k_0}$, we have $g^i(\wh\by_1)=g_*^i(\wh\bx_1)$ for any $i\geq \alpha_0-(n_0+Lk_0)$.
On the other hand, since $\Lambda_{g}=\Lambda_{g_*}$, we also have $g^j(\bx_{g})=g_*^j(\bx_{g_*})$ 
for any $j\geq 0$ and the continuation $\bx_{g}\in \Lambda_{g}$ of $\bx_{g_*}\in \Lambda_{g_*}$.
By this fact together with \eqref{eqn_gg},\eqref{eqn_diamD} and \eqref{eqn_byinD},
\[
\lim _{n\to \infty} 
\frac{1}{n}\sum_{i=0}^{n-1}
\sup_{\by\in D_{g}}\mathrm{dist}\,(g^i(\by), g^i(\bx_{g}))=
\frac{1}{n}\sum_{i=0}^{n-1}
\mathrm{dist}\,(g^i(\wh\by_1), g^i(\bx_{g}))=0.
\]
Since $g$ is taken to be arbitrarily $C^1$-close to $f$, 
it follows that 
$f$ is strongly pluripotent for the continuation $\Lambda_f^{(\mathrm{mj})}$ 
of $\Lambda_{f_0}^{(\mathrm{mj})}$.
This completes the proof of Theorem \ref{mainthm}.
\end{proof}

\appendix

\section{Evaluation of relative slopes}\label{S_relative}
Here we present some elementary results in differential geometry that were used in the main text.

Suppose that $F$ is a $C^1$-surface in $\bb$ compatible with $\boldsymbol{C}_\ve^{\cs}$ and $J$ is a $C^1$-curve parametrized by $\ba:[-t_0,t_1]\too \bb$ for some 
$t_0,t_1>0$ 
satisfying the following conditions.
\begin{itemize}
\setlength{\itemindent}{-16pt}
\item
$\pi_{yz}(J)\subset \pi_{yz}(F)$, where $\pi_{yz}:\bb\too \{0\}\times I_{\ve_0}^2$ is 
the orthogonal projection to the $yz$-plane.
\item
$\ba(0)=\bx_0$ is a point of $F$.
\item
$\ba(t)=(\gamma(t),\iota(t),t)$ for any $t\in [-t_0,t_1]$, where  
$\gamma(t)$ and $\iota(t)$ are $C^1$-functions.
\end{itemize}
Let $U_F=\pi_{yz}(F)$ and let $\vp:U_F\too \rr$ be the $C^1$-function such that, 
for any $(y,z)\in U_F$, $(\vp(y,z),y,z)$ is a point of $F$, in other words, 
$F$ is the graph of $\vp$.
Consider the $C^1$-diffeomorphism $\varPhi:U_F\times I_{\ve_0}^2\too \rr^3$ 
defined as $\varPhi(x,y,z)=(x-\vp(y,z),y,z)$.
Then $\varPhi(F)=U_F$ is contained in the $yz$-plane.
The curve $\wh J=\varPhi(J)$ is parametrized by 
$\wh \ba(t)=\varPhi(\ba(t))=(\wh\gamma(t),\iota(t),t)$ $(-t_0\leq t\leq t_1)$, 
where
\[
\wh\gamma(t)=\gamma(t)-\vp(\iota(t),t).
\]

Here we evaluate the slope $|\wh\gamma'(t)|$ of $\wh\gamma$ at $t$.
For any $\bx\in J$, let $\bx_1$ be the point of $F$ with $\pi_{yz}(\bx)=\pi_{yz}(\bx_1)$, 
and let $\varPhi(\bx)=\wh\bx$, $\varPhi(\bx_1)=\wh\bx_1$. 
We set 
$\theta(\bx)=\angle_{\bx}\la J,\tau_{\bx-\bx_1}(F)\ra$ and 
$\wh\theta(\wh\bx)=\angle_{\wh \bx}\la \wh J,P(\wh \bx)\ra$, 
where $P(\wh\bx)$ is a flat surface in $\rr^3$ passing through $\wh\bx$ and parallel to the $yz$-plane.
Here we consider the case that $\theta(\bx)<(c_0+1)\alpha$ for some constants $c_0>0$ and $0<\alpha<\ve$.
Since $F$ is compatible with $\boldsymbol{C}_\ve^{\cs}$, $\varPhi$ is $O(\ve)$-$C^1$-close to the identity.
It follows that $\wh\theta(\wh\bx)=(1+O(\ve))\theta(\bx)$.
Then, for any $t\in [-t_0,t_1]$, 
\begin{equation}\label{eqn_gammaA}
|\wh\gamma'(t)|=\tan\wh \theta(\wh\gamma(t))<\tan(c_0+1)(1+O(\ve))\alpha<(2c_0+1)\alpha
\end{equation}
if $\ve$ and hence $\alpha$ are sufficiently small.
Here `$2c_0$' is chosen as a positive constant strictly greater than $c_0$.
Since $\|\bx-\bx_1\|=\|\wh\bx-\wh\bx_1\|$ for any $\bx\in J$, we also have 
\begin{equation}\label{eqn_dist_xFA}
\max_{\bx\in J}\mathrm{dist}(\bx,F)<\max(t_0,t_1)(2c_0+1)\alpha.
\end{equation}

\section*{Acknowledgement}
Soma was supported by JSPS KAKENHI Grant Number 22K03342.

\end{document}